\documentclass[11pt,a4paper]{amsart}

\title%
{On Lelong Numbers of Generalized Monge-Ampère Products} %

\author[M. Sera]{Martin L. Sera}
\address{M. Sera, Faculty of Engineering, Kyoto University of Advanced Science, Kyoto 615-8577, Japan}
\email{sera.martin@kuas.ac.jp}

\date{\today}

\dedicatory{In memory of Jean-Pierre Demailly}

\usepackage[initials,alphabetic]{amsrefs} %

\usepackage[UKenglish]{babel}
\usepackage[utf8x]{inputenc}
\usepackage{amsmath,calligra}

\usepackage{amsthm,amssymb,latexsym}
\usepackage[all,cmtip]{xy}
\usepackage{url,ae}
\usepackage{t1enc}
\usepackage{mathrsfs}
\usepackage{graphicx}
\usepackage{geometry}
\usepackage[breaklinks=true, colorlinks=false]{hyperref} %

\usepackage{setspace}
\usepackage{xspace}

\allowdisplaybreaks

\hypersetup{
  pdftoolbar=true,
  pdfmenubar=true,
  pdftitle={On Lelong Numbers of Generalized Monge-Ampère Products},
  pdfauthor={M. Sera},
  pdfsubject={Preprint},
  pdfkeywords={},
  pdfborder= 0 0 .01,
  bookmarksnumbered=false,
  breaklinks=true,   
  colorlinks=true,
  linkcolor={[rgb]{0,.2,.4}},
  citecolor={[rgb]{0,.4,0}},
  menucolor={[rgb]{.4,0,0}},
  urlcolor={[rgb]{.4,0,0}},
}
\makeatletter
\mathcode`l="8000
\lccode`\~=`\l
\DeclareMathSymbol{\lsb@l}{\mathalpha}{letters}{`l}
\lowercase{\gdef~{\ifnum\the\mathgroup=\m@ne \ell \else \lsb@l \fi}}%
\def\section{\@startsection{section}{1}%
  \z@{1\linespacing\@plus.5\linespacing}{\linespacing\@plus .5\linespacing}%
  {\normalfont\bf\centering}}
\makeatother

\usepackage{aliascnt}
\numberwithin{equation}{section}

 \numberwithin{equationX}{section}

\addto\extrasUKenglish{}
\addto\extrasUKenglish{}

\newaliascnt{proposition}{equationX} 
\newtheorem{prop}[proposition]{Proposition}

\newaliascnt{theorem}{equationX} 
\newtheorem{thm}[theorem]{Theorem}

\newaliascnt{lemma}{equationX}  
\newtheorem{lma}[lemma]{Lemma}

\newaliascnt{corollary}{equationX}
\newtheorem{cor}[corollary]{Corollary}

\newtheorem{manualthminner}{Theorem}
\newenvironment{manualthm}[1]{%
  \renewcommand\themanualthminner{#1}%
  \manualthminner
}{\endmanualthminner}
\newtheorem{manualcorinner}{Corollary}
\newenvironment{manualcor}[1]{%
  \renewcommand\themanualcorinner{#1}%
  \manualcorinner
}{\endmanualcorinner}

\theoremstyle{definition}

\newaliascnt{definition}{equationX} 
\newtheorem{df}[definition]{Definition}

\newaliascnt{remdef}{equationX} 
\newtheorem{rem-def}[remdef]{Remark/Def{.}}

\newaliascnt{example}{equationX}
\newtheorem{exa}[example]{Example}

\newaliascnt{remark}{equationX} 
\newtheorem{rem}[remark]{Remark}

\newaliascnt{setting}{equationX} 
\newtheorem{set}[setting]{Setting}

\DeclareMathOperator{\Hom}{\mathscr{H}\text{\kern -3pt {\calligra\Large om}}\,}
\DeclareMathOperator{\Ext}{\mathscr{E}\text{\kern -3pt {\calligra\Large xt}}\,\,}
\DeclareMathOperator{\Image}{\mathscr{I}\text{\kern -3pt {\calligra\Large m}}\,}
\DeclareMathOperator{\Ker}{\mathscr{K}\text{\kern -3pt {\calligra\Large er}}\,}

\newcommand{\PM}{\mathscr{P} \kern -3pt \mathscr{M}}

\newcommand{\CH}{\mathscr{C} \kern -2pt \mathscr{H}}

\newcommand{\CC}{\ensuremath{\mathbb{C}}\xspace}
\newcommand{\PP}{\ensuremath{\mathbb{P}}\xspace}
\newcommand{\PE}{\ensuremath{{\mathbb{P}(E)}}\xspace}

 \newcommand{\minus}{\setminus} %

\newcommand{\eps}{\varepsilon}\newcommand{\ph}{\varphi}

\newcommand{\FS}{{\mathrm{{FS}}}}
\newcommand{\nnef}{{\mathrm{nnf}}}%

\DeclareMathOperator{\codim}{codim}
\DeclareMathOperator{\ord}{ord}

\newcommand{\cf}{cf.\ } \newcommand{\eg}{e.\,g.\ }   \newcommand{\ie}{i.\,e.\ }  %
\newcommand{\ip}{i.\,p.\ }
\newcommand{\st}{s.\,t.\ }

\makeatletter
\def\section{\@startsection{section}{1}%
  \z@{3\linespacing\@plus.5\linespacing}{\linespacing\@plus .5\linespacing}%
  {\normalfont\bf\centering}}
\makeatother

\newcommand{\kdots}{\textrm{,..,\hspace{.2ex}}}

\newcommand{\vtheta}{\vartheta}

\newcommand{\RR}{\mathbb{R}}

\newcommand{\cO}{\mathcal{O}}

\newcommand{\pt}{\mathrm{pt}}
\newcommand{\loc}{\mathrm{loc}}
\newcommand{\smth}{\mathrm{smth}}
\newcommand{\Def}{\mathrm{def}}

\newcommand{\AW}{{\textrm{A\hspace{-2.5pt}\raisebox{.2pt}{W}}}}
\newcommand{\BTD}{\textrm{BTD}\xspace}

\newcommand{\BC}{\ensuremath{\partial\overline\partial}\xspace}
\newcommand{\ddc}{\ensuremath{dd^c}\xspace}

\newcommand{\ignore}[1]{}
\DeclareMathOperator\supp{supp}
\DeclareMathOperator\pr{pr}

\renewcommand{\bar}{\overline}

\newcommand{\smbinom}[2]{\hbox{$\binom{#1}{#2}$}}

\newcommand{\EoPfwEq}{$ $\\[-6ex]}

\def\dimX{\ensuremath{n}\xspace}
\def\dimF{\ensuremath{m}\xspace}
\def\dimY{\ensuremath{n-m}\xspace}
\def\rkE{\ensuremath{r}\xspace}

\newcommand{\ManualItem}[1]{\noindent\makebox[3ex][l]{$\mathrm{(#1)}$}}

\usepackage{ bbm }
\newcommand{\one}{\mathbbm{1}}%

\begin{document}

\begin{abstract}
	We consider generalized (mixed) Monge-Amp\`ere products of quasiplurisub\-harmonic
	functions (with and without analytic singularities) as they were introduced and studied in several articles written by subsets of M.~Andersson, E.~Wulcan, Z.~B\l ocki, R.~L\"ark\"ang, H.~Raufi, J.~Ruppenthal, and the author. We continue these studies and present estimates for the Lelong numbers of pushforwards of such products by proper holomorphic submersions.
	Furthermore, we apply these estimates to Chern and Segre currents of pseudoeffective vector bundles. Among other corollaries, we obtain the following generalization of a recent result by X. Wu. If the non-nef locus of a pseudoeffective vector bundle $E$ on a K\"ahler manifold is contained in a countable union of $k$-codimension\-al analytic sets, and if the $k$-power of the first Chern class of $E$ is trivial, then $E$ is nef.
\end{abstract}

\maketitle

\onehalfspacing
\thispagestyle{empty} 

\section{Introduction}
\noindent
For a plurisubharmonic (psh) function $u$ on a domain $D\subset \CC^n$,
the Lelong number in $x_0\in D$ is defined by
\begin{equation*}
	\nu(u,x_0):=\liminf_{x\rightarrow x_0}\frac{u(x)}{\log\|x-x_0\|}. %
\end{equation*}
This can be seen as a generalization of the vanishing order (of $e^u$) since $\nu(\log |f|,0)=\ord_{x_0} f$ for a holomorphic function $f\colon D\rightarrow \CC$. %
Since its introduction in \cites{Lelong57, Lelong68}, Lelong numbers of psh functions and (its generalized version) of closed positive currents have proven to be very crucial tools in many areas beside of pluripotential theory.
Many of its applications involve estimates on Lelong numbers. %
Let us recall a particular one which gives estimates on Lelong numbers of (mixed) Monge-Ampère products by the Lelong numbers of their potentials from below, see \cites{Demailly85, Demailly93}.

Generalizing a fundamental result from Bedford and Taylor in \cites{BT76, BT82} defining the Monge-Ampère products of bounded psh functions (cited as \autoref{thm:BT-continuity}), Demailly proved in \cites{Demailly85, Demailly93}  that 
for a closed positive current $T$ of bidegree $(p,p)$ and a plurisubharmonic function $u$ such that its unbounded locus $L(u)$ is contained in an analytic set of codimension greater than or equal to $p+k$, 
\begin{equation*}
	(dd^c u)^k\wedge T:=\lim_{\kappa\rightarrow\infty} (dd^cu_\kappa)^k\wedge T
\end{equation*}
is a closed positive current whereby $d^c:=\frac i{2\pi}(\overline{\partial}-{\partial})$ and $u_\kappa$ is a sequence of smooth psh functions decreasing pointwise to $u$ (cited as \autoref{thm:BTD}). %
To shorten the notation, we will use \BTD for the reference to such Monge-Ampère (MA) products. %

With \BTD, we can calculate the Lelong number of a closed positive $(p,p)$-current $T$ by
\begin{equation*}
	\nu(T,x_0)=\int \one_{x_0}(dd^c\log\|x-x_0\|)^{n-p}\wedge T.
\end{equation*}
We get that $\nu(dd^cu,x_0)=\nu(u,x_0)$. %
Demailly's second comparison formula for Lelong numbers (see \cites{Demailly85, Demailly93}) implies 
\begin{equation}\label{eq:IntroMotivation}
	\nu((dd^c u)^k\wedge T,x_0)\geq\nu(u,x_0)^k\cdot\nu(T,x_0)
\end{equation}
for $k+p$ as above. The purpose of this work is to show that such estimates hold for more general Monge-Ampère products as explained in the following.

\medskip

Let $X$ and $Y$ be complex manifolds, and let $\pi\colon X\rightarrow Y$ be a proper holomorphic submersion
with \dimF-dimensional fibres. %
Let $q$ be a quasipsh function, let $\alpha$ be a closed real $(1,1)$-form, and let $T$ be a closed positive $(p,p)$-current, all defined on $X$.
Furthermore, we assume that for any small enough open $V\subset Y$, there is a closed positive%
\footnote{Following the notation usual for currents, we call a $(1,1)$-form $\alpha$ positive if $\alpha\geq 0$.} %
$(1,1)$-form $\gamma$ on $U:=\pi^{-1}(V)$ such that 
$q$ is $(\alpha{+}\gamma)$-psh on $U$, \ie $dd^cq+\alpha\geq -\gamma$. This is always the case if a priori $dd^c q+\alpha$ is assumed to be positive, or if $\pi$ is Kähler, see \autoref{rk:Kaehler-pi}.
Let $q_\kappa$ be a sequence of $(\alpha{+}\gamma)$-psh functions which is decreasing pointwise to $q$.
Following \cite{LRRS}, we obtain that 
for all $k$ such that $\pi(L(q))$ is contained in an analytic set of codimension $\geq k+p-\dimF$,
the current 
\begin{equation}\label{eq:extBTD-def-intro}
	\pi_\ast\big([dd^c q+\alpha]^k\wedge T\big):=\lim_{\kappa\rightarrow\infty} \pi_\ast\big((dd^c q_\kappa+\alpha)^k\wedge T\big)
\end{equation}
is well-defined and locally the difference of two closed positive currents which is independent of the choice of  $\alpha$ and $q_\kappa$, see \autoref{pr:extBTD-def}. 
This generalizes the definition of MA products by \BTD as the $L(q)$ is allowed to be in an analytic set of codimension strictly greater than $k+p$.

\smallskip

The first main result of the present work gives an estimate on the Lelong numbers of currents defined by \eqref{eq:extBTD-def-intro}  which generalizes \eqref{eq:IntroMotivation}.

\newcommand{\FirstTime}{1}
\newcommand{\StatementFirstMainTheorem}{
	Let $\pi\colon X\rightarrow Y$ be a proper holomorphic submersion
	between complex manifolds $X$ and $Y$
	with \dimF-dimensional fibres.
	Fix a point $y\in Y$.
	Let $\theta_1\kdots \theta_t$ be positive $(1,1)$-currents on $X$ such that each $\theta_i$ is in a Kähler class on a neighbourhood of $\pi^{-1}(y)$. %
	Then, there exist positive constants $\delta_i$, for $i=1\kdots t$,
	(which only depend of the Kähler class represented by $\theta_i$ on a neighbourhood of $\pi^{-1}(y)$)
	such that the following statements are correct.

	\ManualItem{i} If the union of all images $\pi(L(\theta_i))$ of the unbounded loci of local \ddc-potentials of $\theta_i$ is contained in an analytic $A$  with $\codim A\geq k_1+\dots +k_t-\dimF$, then
	\begin{equation}\if\FirstTime1\label{eq:mainthm-1a}\else\label{eq:mainthm-1aII}\fi
		\nu\Big(\pi_\ast \big( [\theta_t]^{k_t}\wedge \cdots \wedge[\theta_1]^{k_1}\big), y\Big)
		\geq\prod\big._{i=1}^t \min\{\nu(\theta_i,x),\delta_i\}^{k_i}
	\end{equation}
	for all points $x\in \pi^{-1}(y)$ as long the LHS does not vanish due to the degree of the current.

	\ManualItem{ii} If $\bigcup_i\pi(L(\theta_i))$ is contained in an analytic $A$ with $\codim A\geq k_1+\dots +k_t-t\cdot \dimF$, then
	\begin{equation}\if\FirstTime1\label{eq:mainthm-1b}\else\label{eq:mainthm-1bII}\fi
		\nu\Big(\pi_\ast \big( [\theta_t]^{k_t}\big)\wedge \cdots \wedge\pi_\ast\big([\theta_1]^{k_1}\big), y\Big)
		\geq \prod\big._{i=1}^t \min\{\nu(\theta_i,x),\delta_i\}^{k_i}
	\end{equation}
	for all points $x\in \pi^{-1}(y)$ as long the LHS does not vanish due to the degree of the current.
	
}
\newcommand{\RepeatFirstMainTheorem}{
\begin{manualthm}{\ref{MainThm:general-positive}} %
	\StatementFirstMainTheorem	
	
\end{manualthm}}
\begin{thm}\label{MainThm:general-positive}
	\StatementFirstMainTheorem		
\end{thm}

\begin{rem}\label{rk:MainThmextension}
	(a) The bound $\delta_i$ is determined by whether for each $x\in\pi^{-1}(y)$, there exists a closed positive current  which is in the same cohomology class as $\theta_i$ near $\pi^{-1}(y)$, smooth everywhere except in isolated points, and whose Lelong number is greater than or equal to $\delta_i$ in $x$; see \autoref{lm:peaker}.
	For example, if $\pi$ is projective %
	and $\theta_i$ is in the class induced by the Fubini-Study metric, then we can select $\delta_i=1$. %

	\ManualItem{b} If we consider proper holomorphic submersions $\pi_i\colon X_i\rightarrow Y$, $i=1\kdots t$, and if each $\theta_i$ is defined on $X_i$ with analogous properties as in \autoref{MainThm:general-positive}, there are positive constants $\delta_i$ such that
	\begin{equation}%
		\nu\Big(\pi_{t,\ast} \big( [\theta_t]^{k_t}\big)\wedge \cdots \wedge\pi_{1,\ast}\big([\theta_1]^{k_1}\big), y\Big)
		\geq \prod\Big._{i=1}^t \min\{\nu(\theta_i,x_i),\delta_i\}^{k_i}
	\end{equation}
	for all points $x_1\in\pi_1^{-1}(y)\kdots x_t\in\pi_t^{-1}(y)$ and suitable $k_1\kdots k_t$.
\end{rem}

\smallskip

Let us sketch the proof of \autoref{MainThm:general-positive} (i).
We will use generalized Monge-Ampère products for currents with analytic singularities which were introduced and studied in \cites{Andersson05,AW14-MA,ASWY17,ABW,LRSW,Blocki19-quasiMA,LSW} and the present work as explained as follows. %

If a quasipsh function $q$ equals locally $c\log\|F\|^2+b$ for a positive constant $c$, holomorphic tuple $F$ and bounded function $b$, $q$ is said to have analytic singularities in $\{F=0\}$ (given as reduced analytic set). If moreover $b$ is smooth, $q$ has so-called neat analytic singularities.
We say that a closed quasipos $(1,1)$-current $\theta$ has (neat) analytic singularities in $Z$ if it is locally given by $\theta=dd^c q$ for a quasipsh $q$ with (neat) analytic singularities in $Z$.
For $i=1\kdots t$, let $\theta_i$ be closed quasipos $(1,1)$-currents with analytic singularities in $Z_i$, and let $\alpha_i$ be (closed) real $(1,1)$-forms. Then, the \emph{generalized Monge-Ampère} product $[\theta_t]_{\alpha_t}^{k_t}\wedge \cdots \wedge[\theta_1]_{\alpha_1}^{k_1}$ for any $k_1\kdots k_t$ is defined recursively by
\begin{equation}\label{eq:IntroDefGenMA}
	[\theta_i]_{\alpha_i}\wedge T:=\theta_i\wedge\one_{Z^c} T+\alpha_i\wedge\one_{Z} T.
\end{equation}
This definition works also if we replace $\alpha_i$ by quasipositive $(1,1)$-currents $\eta_i$ with neat analytic singularities in isolated points since the last term in \eqref{eq:IntroDefGenMA} is defined by \BTD, see \autoref{def:genMA}.
This %
will be useful to get estimates for the Lelong number of such MA products, see \autoref{pr:Lelong-estimate-AW-products} and \autoref{rk:Lelong-estimate-AW-products-needs-eta}.
Although the generalized Monge-Ampère product lacks some properties of a product, it is a suitable extension of the \BTD Monge-Ampère product to arbitrary degrees independent of the unbounded locus of $q$. By \autoref{thm:extBTD-coincides-genMA} (\cf \cite{LRSW}*{Thm~1.2}), we get that if $\eta_i$ are in the same \BC-class as $\theta_i$, and if $\bigcup_i\pi(L(\theta_i))$ is contained in an analytic $A$ with $\codim A\geq k_1+\dots +k_t- \dimF$, then %
\begin{equation}\label{eq:extBTDgMaIntro}
	\pi_\ast \big( [\theta_t]^{k_t}\wedge \cdots \wedge[\theta_1]^{k_1}\big)=\pi_\ast\big([\theta_t]_{\eta_t}^{k_t}\wedge \cdots \wedge[\theta_1]_{\eta_1}^{k_1}\big).
\end{equation}

Therefore, the idea behind the proof of \autoref{MainThm:general-positive} is to use Demailly's regularization \cite{Demailly92-reg-closed-pos-currents}*{Main Thm~1.1} (see also \autoref{thm:DemRegAnalySing}) which gives us that the (quasi-) positive $(1,1)$-currents $\theta_i$ can be approximated by sequences of quasipositive $(1,1)$-currents $\theta_{i,\kappa}$ with analytic singularities (and other suitable properties). %
By the monotone continuity result \autoref{thm:extBTD-continuity}, %
we get %
\begin{equation}\label{eq:extBTDContIntro}%
	\pi_\ast\big([\theta_t]^{k_t}\wedge \cdots \wedge[\theta_1]^{k_1}\big)%
	:=\lim_{\kappa\rightarrow\infty}
	\pi_\ast\big([\theta_{t,\kappa}]^{k_t}\wedge \cdots \wedge[\theta_{1,\kappa}]^{k_1}\big)
\end{equation} %
(actually for any approximation). %
We would like to add that \autoref{thm:extBTD-continuity} implies
\autoref{cr:extBTD-continuity-special} which asserts that for certain $k_1\kdots k_t$,
\pagebreak[0]
\begin{equation*}%
	\pi_\ast [\theta_t]^{k_t}\wedge \cdots \wedge\pi_\ast[\theta_1]^{k_1}
	:=\lim_{\kappa\rightarrow\infty}
	\pi_\ast[\theta_{t,\kappa}]^{k_t}\wedge \cdots \wedge\pi_\ast[\theta_{1,\kappa}]^{k_1}.
\end{equation*} %
We obtain that \autoref{MainThm:general-positive} is implied by
\eqref{eq:extBTDContIntro}, \eqref{eq:extBTDgMaIntro} and
the following theorem, which is the second main result of the present work. %

\newcommand{\StatementSecondMainTheorem}{
	Let $\pi\colon X\rightarrow Y$ be a proper holomorphic submersion
	between complex manifolds $X$ and $Y$
	with \dimF-dimensional fibres.
	Fix a point $y\in Y$.
	Let $\theta_1\kdots \theta_t$ be positive $(1,1)$-currents with analytic singularities on $X$ such that each $\theta_i$ is in a Kähler class on a neighbourhood of $\pi^{-1}(y)$. %
	Then, there exist positive constants $\delta_i$, $i=1\kdots t$,
	(which only depend of the Kähler class represented by $\theta_i$ on a neighbourhood of $\pi^{-1}(y)$)
	such that the following is correct.

	\ManualItem{i} \emph{Existence of $\eta_i$:}
	For every $x\in \pi^{-1}(y)$, 
	there exist closed quasipos $(1,1)$-currents $\eta_{i}=\eta_{i,x}\in\{\theta_i\}_{\BC}$, positive near $\pi^{-1}(y)$, with neat analytic singularities only in $\{x\}$ %
	(and %
	with $\nu(\eta_i,x)= \delta_i$) such that
	\begin{equation}\if\FirstTime1\label{eq:mainthm-2a}\else\label{eq:mainthm-2aII}\fi
		\nu\Big(\pi_\ast \big( [\theta_t]_{\eta_t}^{k_t}\wedge \cdots \wedge[\theta_1]_{\eta_1}^{k_1}\big), y\Big)
		\geq\ \prod\big._{i=1}^t \min\{\nu(\theta_i,x),\delta_i\}^{k_i}%
	\end{equation}
	for all $m\leq k_1+\ldots+k_t\leq\dim X$.
	Thereby, 	$\eta_i=\eta_{i,x}$ depends only of the Kähler class of $\theta_i$ near $\pi^{-1}(y)$ and $x$.

	\ManualItem{ii} \emph{Independency of $\eta_i$:}
	For $i=1\kdots t$, let $\eta_i\in\{\theta_i\}_{\BC}$ be closed quasipos $(1,1)$-currents  with neat analytic singularities in isolated points on $X$ such that there is a point $x\in X$ %
	with $L(\eta_i)\cap\pi^{-1}(y)=\{x\}$ and $\nu(\eta_i,x)\leq\delta_i$. Then,
	\begin{equation}\if\FirstTime1\label{eq:mainthm-2b}\else\label{eq:mainthm-2bII}\fi
		\nu\Big(\pi_\ast \big( [\theta_t]_{\eta_t}^{k_t}\wedge \cdots \wedge[\theta_1]_{\eta_1}^{k_1}\big), y\Big)
		\geq\ \prod\big._{i=1}^t \min\{\nu(\theta_i,x),\nu(\eta_i,x)\}^{k_i}%
	\end{equation}
	for all $m\leq k_1+\ldots+k_t\leq\dim X$.
}
\newcommand{\RepeatSecondMainTheorem}{
\begin{manualthm}{\ref{MainThm:analytic-singularities}} %
	\StatementSecondMainTheorem	
\end{manualthm}}
\begin{thm}\label{MainThm:analytic-singularities}
	\StatementSecondMainTheorem
\end{thm}
\renewcommand{\FirstTime}{0}

\autoref{rk:MainThmextension} (a) applies to \autoref{MainThm:analytic-singularities}, as well.

The idea to consider $\eta_i$ with singularities in isolated points is inspired by X. Wu's approach to choose specific non-smooth approximations of Segre currents to obtain Lelong number estimates of them, see \cite{WuX22}*{Proof of Thm~2}.

Following the arguments in the proof of \cite{LRSW}*{Thm~1.1 (3)}, the Lelong numbers of $\pi_\ast\big([\theta_t]_{\alpha_t}^{k_t}\wedge \cdots \wedge[\theta_1]_{\alpha_1}^{k_1}\big)$ are independent of $\alpha_i$ among closed forms in the same class as $\theta_i$. Unfortunately, this is not correct if we replace $\alpha_i$ by $\eta_i$ with different kind of isolated singularities, see \autoref{rk:pushforward-Lelong-not-indep-of-eta}.
By \autoref{pr:pushforward-Lelong-indep-of-sp-eta}, we get at least that
$\pi_\ast \big( [\theta_t]_{\eta_t}^{k_t}\wedge \cdots \wedge[\theta_1]_{\eta_1}^{k_1}\big)$'s Lelong numbers are independent of $\eta_i$ among currents with the same kind of singularities in all points.

\bigskip

By applying \autoref{MainThm:analytic-singularities} to Segre currents defined by \eqref{eq:DefAWSegre1F} as in \cite{LRSW}*{Thm~1.1}, %
we obtain the following corollary; %
see \autoref{rdf:SegreCurrents-analySing}.

\begin{cor}\label{cr:application-analy-sing}
	Let $X$ be a complex manifold, let $E\rightarrow X$ be a (pseudoeffective) holomorphic vector bundle on $X$ of rank \rkE, and
	let $e^{-\ph}$ be a semipositive singular metric on $\cO_{\PP(E)}(1)$ with analytic singularities, \ie the weights $\ph$ are psh with analytic singularities.
	Then, for every $x\in X$, $\xi\in \pi^{-1}(x)$ and every $k\leq \dim X$, there exists a closed real $(k,k)$-current $S=S_\xi$ in the Segre class $(-1)^ks_{k}(E)$ such that (i) locally $S$ is the difference of two closed positive currents, (ii) $S$ is positive close to $x$, and (iii) 
	its Lelong number can be estimated from below by \[\nu(S,x)\geq \min\{\nu(\ph,\xi), 1\}^{k+tr-t}.\] %
\end{cor}

Analogously, \autoref{MainThm:general-positive} gives us estimates for the Lelong numbers of Segre currents defined by \autoref{thm:def-SegreCurrents-general}, following \cite{LRRS}*{Prop.~4.6}.%
\begin{cor}\label{cr:appl-general-positive}
	Let $X$ be a complex manifold of dimension \dimX,
	let $E\rightarrow X$ be a (pseudoeffective) holomorphic vector bundle on $X$ of rank \rkE, and
	let $e^{-\ph}$ be a semipositive singular metric on $\cO_{\PP(E)}(1)$ such that $\pi(L(\ph))$ is contained in an analytic set of codimension $s$.
	Then, for all partitions $k_1+\ldots+ k_t=k\leq s$, all $x\in X$ and all $\xi\in \pi^{-1}(x)$, we get
	\begin{equation*}
		(-1)^k\nu(s_{k_t}(E,\ph)\wedge \cdots \wedge s_{k_1}(E,\ph),x)
		\geq
		\min\{\nu(\ph,\xi),1 \}^{k+tr-t}. %
	\end{equation*}
\end{cor}

Furthermore, \autoref{cr:appl-general-positive} implies the following sufficient condition for $E$ being nef.
\newcommand{\StatementThirdIntroCor}{
	Let $X$ be a compact complex manifold of dimension \dimX,
	let $E\rightarrow X$ be a (pseudoeffective) holomorphic vector bundle on $X$,
	and
	let $e^{-\ph}$ be a semipositive singular metric on $\cO_{\PP(E)}(1)$ such that $\pi(L(\ph))$ is contained in an analytic set of codimension $s$.
	If there are $k_1\kdots k_t$ with $k_1+\ldots+k_t\leq s$ and $s_{k_1}(E)\cdots s_{k_t}(E)=0$, then $E$ is nef.
}
\newcommand{\RepeatThirdIntroCor}{
\begin{manualcor}{\ref{cr:positive-Segre-current-in-vanishing-class-gives-nef}} %
	\StatementThirdIntroCor	
\end{manualcor}}
\begin{cor}\label{cr:positive-Segre-current-in-vanishing-class-gives-nef}
	\StatementThirdIntroCor
\end{cor}

Let us point out that $X$ does not need to be Kähler in \autoref{cr:positive-Segre-current-in-vanishing-class-gives-nef}.
Yet, the condition on $\pi(L(\ph))$, which gives us the existence of the positive Segre currents, can be seen as quite strong.
Assuming that $X$ is Kähler, we can weaken this condition by using the non-nef locus %
which is defined by
\begin{equation*}
	L_\nnef(E):=\pi\big(L_\nnef(\cO_{\PE}(1))\big)=\pi\Big(\bigcup\big._{\delta>0} \bigcap\big._{\theta}\,E_{+}(\theta)\Big)
\end{equation*}
where the intersection runs over all closed $\delta\omega_\PE$-positive $(1,1)$-currents $\theta\in\{\omega_\PE\}=\cO_{\PE}(1)$ for a Kähler form $\omega_\PE$ on \PE, and  $E_+(\theta):=\{\xi\in\PE\colon\nu(\theta,\xi)>0 \}$. %
We obtain the following generalization of X.~Wu's main result in \cite{WuX22}. %

\newcommand{\StatementFourthIntroCor}{
	Let $(X,\omega)$ be a compact Kähler manifold, %
	and 	let $E$ be a pseudoeffective vector bundle on $X$ %
	such that $L_\nnef(E)$ is contained in a countable union of analytic sets of codimension $s$. %
	If there is a $k\leq s$ with  $c_1(E)^k = ({-}s_1(E))^k = 0$, %
	then $E$ is nef.
}
\newcommand{\RepeatFourthIntroCor}{
\begin{manualthm}{\ref{thm:vanishing-class-and-more-give-nef}} %
	\StatementFourthIntroCor	
\end{manualthm}}
\begin{thm}\label{thm:vanishing-class-and-more-give-nef}
	\StatementFourthIntroCor		
\end{thm}
If $X$ is strongly pseudoeffective, \ie  $L_\nnef(E)\neq X$ (see \cite{BDPP13}*{Def.~7.1}), then the assumption on  $L_\nnef(E)$ is satisfied for $k=1$. Hence, we obtain X. Wu's original result that all strongly pseudoeffective vector bundles with trivial first Chern class are nef.
To prove \autoref{thm:vanishing-class-and-more-give-nef}, we will follow X. Wu's argumentation. Yet, instead of the analytical construction of closed positive currents tailored to prove his result
(we will see that these coincide with the one defined above, see \autoref{rk:Wu's-currents-are-the-same}), we will use \autoref{MainThm:general-positive} which is based on the calculus of generalized Monge-Ampère products, see Sections \ref{sc:currents-analy-sing} and \ref{sc:Lelong}.

\bigskip

The present work is organized as follows.
In \autoref{sc:preliminaries}, we recall and elaborate some basic properties which are later used.
We continue with the study of currents as they are defined by \eqref{eq:extBTD-def-intro} in \autoref{sc:pushforwards}.
In particular, we prove the mentioned monotone continuity results \autoref{thm:extBTD-continuity} and \autoref{cr:extBTD-continuity-special} there.
In \autoref{sc:currents-analy-sing}, we introduce so-called currents with analytic singularities and study their properties.
We continue with a section about Lelong numbers of closed real currents, see \autoref{sc:Lelong}.
\autoref{sc:proofs} is dedicated to prove the main results, Theorems \ref{MainThm:general-positive} and \ref{MainThm:analytic-singularities}.
We conclude the article with a short discussion on Segre currents of (pseudoeffective) vector bundles and proving \autoref{cr:positive-Segre-current-in-vanishing-class-gives-nef} and \autoref{thm:vanishing-class-and-more-give-nef} in \autoref{sc:Segre}.

\subsection*{Acknowledgements}
The author is very grateful to Elizabeth Wulcan and the unknown referee for their valuable comments and suggestions which helped to improve the article a lot. The author was supported by JSPS KAKENHI Grant Number JP20K14319.

\section{Preliminaries}
\label{sc:preliminaries}
\noindent
In this section, we will introduce some notations and collect results from the literature.
Considering the wedge product of pushforwards, the following tool is very helpful.
\begin{lma}(Lem.~6.3 in \cite{LRSW})\label{lm:fibre-products-and-direct-images}
	Let $Y$ be a complex manifold.
	For $i=1\kdots t$,
	let $X_i$ be complex manifolds, let $\pi_i\colon X_i\rightarrow Y$ be proper holomorphic submersions, and let $\alpha_i$ be forms on $X_i$.
	Let $\tilde X:= X_t \times_Y \cdots \times_Y X_1\overset{\varpi}\longrightarrow Y$ be the fibre product of $\pi_t\kdots \pi_1$, let $\pr_i\colon \tilde X\rightarrow X_i$ denote its projections on the $i$th component,
	and let $T$ be a current on $X_1$.
	Then,
	\begin{equation*}
		\pi_{t,\ast}\alpha_t\wedge \cdots \wedge \pi_{2,\ast} \alpha_2\wedge \pi_{1,\ast} T
		=\varpi_\ast\big(\pr_t^\ast\alpha_t\wedge \cdots \wedge \pr_2^\ast\alpha_2\wedge \pr_1^\ast T\big).
	\end{equation*}
\end{lma}

\bigskip

\subsection{Closed positive currents}
In this subsection, we recall some basic facts about closed positive currents. 
Throughout it, let $X$ denote a complex manifold of dimension \dimX.
In \cite{LRRS}*{Def.~4.2}, we introduced the following notion.
\begin{df} %
	A strongly positive $(k,k)$-form $\beta$ is called \emph{bump form at a point} $x\in X$ if for a (or equivalently for any) Kähler form $\omega$ defined near $x$, there exists a constant $\delta > 0$ such that $\beta-\delta\omega^k$ is a strongly positive form in a neighbourhood of $x$.
\end{df}
\begin{rem}[Lem.~7.2 in \cite{LRSW}]\label{rk:BumpFormExists}
	As presented in the proof of \cite{LRRS}*{Lem.~4.3}, for every analytic $A$ with $\dim A\leq k$ and every point $x\in A$, there exists a $(k,k)$-bump form $\beta$ with arbitrarily small support \st %
	the support of $dd^c\beta$ is in the complement of $A$.
\end{rem}
The notion of bump form turns out to be very useful to obtain the uniqueness of the extensions of closed positive currents across analytic sets as follows. %

\begin{lma}[Lem.~4.5 in \cite{LRRS}]\label{lm:Unique} %
	Let $T$ and $S$ be two closed positive $(p,p)$-currents on $X$ such that $T=S$ on $X\minus A$ for an analytic set $A$ with $\codim A\geq p$. %
	If for every point $x \in A$, there is an $(n{-}p,n{-}p)$-bump form $\beta$ at $x$ with arbitrarily small support such that
	$\int T \wedge \beta = \int S \wedge \beta$,
	then $T=S$ on $X$.
\end{lma}

\begin{prop}[\cf Rem.~4.7 in \cite{LRRS}]\label{pr:convergence-of-positive-currents}
	Let $A$ be an analytic subset of $X$ with $\codim A$ $\geq p$, and
	let $T_\kappa$ be a sequence of closed positive $(p,p)$-currents on $X$. %
	We assume (i) the sequence $T_\kappa$ converges weakly to a closed positive current $T'$ on $X\minus A$ , and
	(ii) $\int T_\kappa\wedge\beta$ converges to numbers $T_\beta$ for any $(n{-}p,n{-}p)$-bump form $\beta$ with arbitrarily small support as $\kappa\rightarrow\infty$.
	Then, $T_\kappa$ converges weakly to a closed positive $(p,p)$-current $T$ on $X$ which is uniquely defined by $T'$ and $T_\beta=\int T\wedge \beta$.
\end{prop}
\begin{proof}
	We follow the argumentation of the proof of Prop.~4.6 in \cite{LRRS}.

	Due to the assumptions (i) and (ii), we obtain that the trace measures of $T_\kappa$ are (locally) uniformly bounded in $\kappa$. Therefore, the Banach-Alaoglu theorem implies that there is a subsequence $\kappa_\lambda\rightarrow\infty$ such that $T_{\kappa_\lambda}$ converges weakly to a closed positive current $T$.
	Let us assume that there is another subsequence $T_{\kappa'_\lambda}$ which does not converge to $T$. By passing to a subsequence, we may assume that $T_{\kappa'_\lambda}$ converges to a closed positive current $S$ different from $T$.
	By the assumption (i), we get that $S$ equals $T$ on $X\minus A$. Furthermore, the assumption (ii) implies that $\int T\wedge \beta = T_\beta= \int S\wedge\beta$ for any $(n{-}p,n{-}p)$-bump form with arbitrarily small support. By \autoref{lm:Unique}, we get $T=S$ on $X$. This is a contradiction.
\end{proof}

\medskip

Next we recall the fundamental result of Bedford-Taylor defining MA products for bounded psh functions and its generalization by Demailly.

\begin{thm}[\cf \cites{BT76,BT82}]\label{thm:BT-continuity}
	Let $u_1\kdots u_t$ be bounded psh functions, and let $T$ be a closed positive $(p,p)$-current on $T$.
	Then, there exists a well-defined closed positive $(t+p,t+p)$-current
		$$dd^cu_t\wedge\dots\wedge dd^cu_1\wedge T,$$
	locally given as the limit of
	$ dd^cu_{t,\kappa}\wedge\dots\wedge dd^cu_{1,\kappa}\wedge T$
	for sequences of smooth psh functions $u_{i,\kappa}$ decreasing pointwise to $u_i$. %
	Furthermore, we obtain that this is monotone continuous: Let $v_{i,\kappa}$ be sequences of (not necessarily smooth) psh functions decreasing pointwise to $u_i$. %
	Then,
	\begin{align*}
		u_tdd^cu_{t-1}\wedge\dots\wedge dd^cu_1\wedge T
		&= \lim_{\kappa\rightarrow\infty}v_tdd^c v_{t-1,\kappa}\wedge\dots\wedge dd^cv_{1,\kappa}\wedge T,
		\makebox[0em][l]{\quad and}\\
		dd^cu_t\wedge\dots\wedge dd^cu_1\wedge T
		&= \lim_{\kappa\rightarrow\infty}dd^c v_{t,\kappa}\wedge\dots\wedge dd^cv_{1,\kappa}\wedge T.
	\end{align*}
\end{thm}

\begin{df}\label{df:ULocus} 
	For a psh function $u$ on $X$, the \textit{unbounded locus} $L(u)$ is defined as the closed set of all points $x\in X$ such that $u$ is unbounded near $x$. In general, $L(u)$ is strictly larger than the pole set of $u$.
\end{df}

\begin{thm}[\cf \cites{Demailly85, Demailly93}; shortly \emph{\BTD}]\label{thm:BTD} %
	Let $u$ be a psh function, and let $T$ be a closed positive $(p,p)$-current. If $L(u) \cap \supp T$ is contained in an analytic $A$ with $\codim A\geq k+p$, then there exists a well-defined closed positive $(k{+}p,k{+}p)$-current $(dd^cu)^k\wedge T$
	(locally) given as the limit of
		$$ (dd^c u_\kappa)^k \wedge T$$
	for any sequence of smooth psh functions $u_\kappa$ decreasing pointwise to $u$.
	Furthermore, we obtain that this is monotone continuous:
	Let $u_1\kdots u_t$ be psh functions with $L(u_i) \cap \supp T$ is contained in an analytic $A$ with $\codim A\geq t+p$, and $v_{i,\kappa}$ be sequences of (not necessarily smooth) psh functions decreasing pointwise to $u_i$. %
	Then,
		$$ dd^cu_t\wedge\dots\wedge dd^cu_1\wedge T = \lim_{\kappa\rightarrow\infty}dd^c v_{t,\kappa}\wedge\dots\wedge dd^cv_{1,\kappa}\wedge T.$$
\end{thm}

Considering the mixed Monge-Ampère product (with different potentials in each factor),
the codimension condition can be weaken by considering the (Hausdorff) dimension of the intersections of the unbounded loci, see Thm~4.5 and Prop.~4.9 in \cite{DemaillyAG}*{Chp.~III}.

\bigskip

\subsection{Closed quasipositive currents}
Let $X$ be a complex manifold of dimension \dimX, and let $\gamma$ be a (positive) $(1,1)$-form on $X$.
We call a function $q$  on $X$ $\gamma$-psh if $dd^c q +\gamma$ is positive. This is equivalent to $q$ is quasipsh with $dd^c q+\gamma\geq 0$ in the sense of currents.
Let $\theta$ be a closed \emph{quasipositive} $(1,1)$-current on $X$, \ie locally, $\theta =dd^c q$ for a quasipsh function $q$ (shortly \emph{quasipos}; also called almost positive).
$\theta$ is called $\gamma$-\emph{positive} ($\gamma$-\emph{pos}) if $\theta+\gamma\geq 0$ in the sense of currents, \ie locally, $\theta =dd^c q$ for a $\gamma$-psh function $q$.
Let us pick a closed real $(1,1)$-form $\alpha$ which is in the same \BC-cohomology class as the $\gamma$-pos $\theta$,
\ie $\theta-\alpha=dd^c q$ for a $(\gamma{+}\alpha)$-psh function $q$ on $X$.
We define $L(\theta):=L(q)$ which is independent of the choice of $q$ and $\alpha$.
Let $q_\kappa$ be a sequence of (smooth) $(\gamma{+}\alpha)$-psh functions which is decreasing pointwise to $q$.
Set $\theta_\kappa:=dd^c q_\kappa+\alpha$ which are (smooth) $\gamma$-positive currents.
Obviously, the sequence $\theta_\kappa$ converges weakly to $\theta$. Let us use the following notation improperly by saying that the sequence $\theta_\kappa$ is \emph{decreasing pointwise to} $\theta$. This is motivated by the following.
If $u$ is a psh functions on a small enough open set with $dd^c u =\theta +\gamma\geq 0$ (assuming that $\gamma$ is closed), then $u_\kappa:=q_\kappa-q+u$ are (smooth) psh functions with $dd^c u_\kappa=\theta_\kappa+\gamma\geq 0$ such that the sequence $u_\kappa$ is decreasing pointwise to $u$. 

\smallskip

Locally, we may assume that $\gamma$ is closed and positive: On any small enough open set, there is a Kähler form $\omega$ and a constant $C$ such that $\gamma\leq C\omega$. In particular, $\theta+C\omega\geq\theta+\gamma\geq 0$ (and $\theta_\kappa+C\omega\geq 0$). Therefore, we may replace $\gamma$ by $C\omega$ on any small enough open set.

\smallskip

Let $T$ be a closed positive $(p,p)$-current, and assume that $\theta_\kappa$ are smooth.
By \BTD, for all $k$ such that $L(\theta)$ is contained in an analytic set of codimension $\geq k+p$, %
 (assuming $\gamma$ is closed, see above) %
\begin{equation*}
	(\theta+\gamma)^k \wedge T\overset{\Def}{=} \lim_{\kappa\rightarrow\infty} (\theta_\kappa+\gamma)^k\wedge T\overset{\loc}=\lim_{\kappa\rightarrow\infty} (dd^c u_\kappa)^k\wedge T
\end{equation*}
is a closed positive current which is independent of the choice of $q_\kappa$ / $\theta_\kappa$ / $u_\kappa$.
We can extend this definition to (wedge) powers of the quasipositive $\theta$ by
\pagebreak[0]
\begin{equation}\label{eq:BTD-for-quasi}
	\begin{split}
	\theta^k\wedge T &= \big(\theta+\gamma -\gamma\big)^k \wedge T:=\sum\big._{i=0}^k (-1)^i \smbinom{k}{i} \gamma^{i}\wedge (\theta+\gamma)^{k-i}\wedge T\\
	& \overset{\Def}{=} \lim_{\kappa\rightarrow\infty} \sum\big._{i=0}^k (-1)^i \smbinom{k}{i} \gamma^{i}\wedge (\theta_\kappa+\gamma)^{k-i}\wedge T = \lim_{\kappa\rightarrow\infty} \theta_\kappa^k\wedge T .
\end{split}\end{equation}
Due to the last equation, we see that the definition is independent of the choice of $\gamma$.
$\theta^k\wedge T$ is a closed real $(k{+}p,k{+}p)$-current which is locally the difference of two closed positive currents.

\medskip

At the end of this section,
we recall following versions of Demailly's regularization of quasipos $(1,1)$-currents fitting to our applications. %

\begin{thm}[\cite{Demailly82}, Main Thm~1.1 in \cite{Demailly92-reg-closed-pos-currents}]\label{thm:Dem-smooth-approximation}
	Let $X$ be a compact complex manifold, and fix a Hermitian form $\omega$ and a real $(1,1)$-form $\gamma$ on $X$.
	Let $\theta$ be a $\gamma$-positive $(1,1)$-current on $X$. Then,
	there are a constant $C$ only dependent of $(X,\omega)$, continuous functions $\lambda_\kappa$ and $(\gamma{+}C\lambda_\kappa\omega)$-pos $(1,1)$-forms $\theta_\kappa\in\{\theta\}_{\BC}$ such that as $\kappa\rightarrow\infty$, %
	(i)  the sequence $\theta_\kappa$ is decreasing pointwise to $\theta$, and %
	(ii) $\lambda_\kappa(x)\searrow\nu(\theta,x)$ for all $x\in X$.
\end{thm}

\begin{thm}[Main Thm~1.1 in \cite{Demailly92-reg-closed-pos-currents}]
	\label{thm:DemRegAnalySing}
	Let $X$ be a complex manifold, and fix a Hermitian form $\omega$ and a real $(1,1)$-form $\gamma$ on $X$.
	Let $\theta$ be a $\gamma$-positive $(1,1)$-current on $X$. Then, for every relatively compact $U\Subset X$,
	there are positive constants $\eps_\kappa$ and $(\gamma{+}\eps_\kappa\omega)$-pos $(1,1)$-currents $\theta_\kappa\in\{\theta\}_{\BC}$ with analytic singularities on $U$ such that as $\kappa\rightarrow\infty$, %
	(i) the sequence $\theta_\kappa$ is decreasing pointwise to $\theta$, %
	(ii) $\nu(\theta_\kappa,x)\nearrow\nu(\theta,x)$ for all $x\in U$ (uniformly), and %
	(iii) $\eps_\kappa\rightarrow 0$. %
\end{thm}
In its original version, $X$ is assumed to be compact and $\theta_\kappa$ may not have analytic singularities (instead $\nu(\theta_\kappa,x)=\nu(\theta,x)$). Nevertheless, the proof also works in the relatively compact setting as explained in the paragraph after \cite{Demailly92-reg-closed-pos-currents}*{Main Thm~1.1}.
Furthermore, \cite{Demailly92-reg-closed-pos-currents}*{Rem.~5.15} ensures the existence of $\theta_\kappa$ with analytic singularities (see \autoref{df:QPshAS}) by weakening the condition on $\theta_\kappa$'s Lelong numbers satisfying (ii). %

\section{Pushforwards of MA products by submersions}%
\label{sc:pushforwards}
\noindent
Following \cite{LRRS}*{Sec.~4 \& Prop.~5.3},
we can define Monge-Ampère products for higher degrees than in \BTD if one considers the pushforwards by proper holomorphic submersions. %

\begin{prop} %
\label{pr:extBTD-def}
	Let $\pi\colon X\rightarrow Y$ be a proper holomorphic submersion between complex manifolds $X$ and $Y$ with \dimF-dimensional fibres, %
	let $T$ be a closed positive $(p,p)$-current on $X$, and
	let $\theta$ be a closed $\gamma$-pos $(1,1)$-current %
	for a (positive) $(1,1)$-form $\gamma$
	such that for any small enough open $V\subset Y$,
	there is a \emph{closed} positive $(1,1)$-form $\gamma_{+}$ on $\pi^{-1}(V)$ with $\gamma\leq\gamma_{+}$  (for example, $\pi$ is Kähler, see \autoref{rk:Kaehler-pi}).
	Let $\theta_{\kappa}\in\{\theta\}_{\BC}$ be closed $\gamma$-positive $(1,1)$-forms
	such that the sequence $\theta_\kappa$ is decreasing pointwise to $\theta$,
	\ie locally there are $\gamma$-psh potentials of $\theta_{\kappa}$ decreasing pointwise to a $\gamma$-psh potential of $\theta$; and let $k$ be a positive integer.
	If $\pi(L(\theta))$ is contained in analytic $A$ with $\codim A\geq k+p-m$, then
	\begin{equation*}
		S:=\pi_\ast \big( [\theta]^{k} \wedge T \big):= \lim_{\kappa\rightarrow\infty} \pi_\ast \big( \theta_\kappa^{k} \wedge T\big)
	\end{equation*}
	is a closed real $(k{+}p,k{+}p)$-current on $Y$, %
	independent of the choice of $\theta_\kappa$, %
	and locally\footnote{to be more precise, on all $V$ from before} the difference of two closed positive currents.
	If $\theta$ is positive, then $S$ is positive. %
\end{prop}

	We would like to stress that in general, $[\theta]^k\wedge T$ is not defined as a current on $X$ for $\codim A-p<k\leq \codim A-p+\dimF$ in the setting above.  
	Later, we will see that it can be defined for currents with analytic singularities, see \autoref{sc:currents-analy-sing}.

\begin{rem}\label{rk:Kaehler-pi}
	Following \cite[Def.~4.1]{Fujiki78}, %
	a holomorphic function $\pi\colon  X\rightarrow Y$ is called Kähler if there is an open covering $\{U_\kappa\}$ of $X$ and plurisubharmonic functions $\ph_\kappa$ on $U_\kappa$ \st $\ph_\kappa$ is strictly plurisubharmonic on $\pi^{-1}(y)\cap U_\kappa$ and $\ph_\kappa-\ph_\lambda$ is pluriharmonic on $\pi^{-1}(y)\cap U_\kappa \cap U_\lambda$ for all $y\in Y$.
	In case of that $\pi$ is a locally trivial proper holomorphic submersion (which is equivalent to all fibres of $\pi$ are biholomorphic to each other, \cf \cite{FischerGrauert65}), we get $\pi$ is Kähler if its fibres are Kähler.
	In general, %
	deformations of Kähler manifolds are not necessarily Kähler morphisms,
	see \cite[Exp.~3.9]{Bingener83-II} due to Deligne.
	
	By \cite[Lem.~4.4]{Fujiki78}, %
	the preimages $U=\pi^{-1}(V)$ of a proper Kähler $\pi\colon  X\rightarrow Y$ are Kähler for any small enough open $V\subset Y$, \ie there is a Kähler form $\omega_{U}$ on $U$.
	Moreover, for any $(1,1)$-form $\gamma$ on $X$ and any small enough open $V$, there exists a constant $C$ such that $\gamma_+:=C\omega_U\geq \gamma$ with closed positive form $\gamma_+$ on $U=\pi^{-1}(V)$. Analogously, every closed $(1,1)$-form $\alpha$ can be decomposed as $\alpha=\alpha_+-\alpha_-$ with closed positive forms $\alpha_\pm$ on $\pi^{-1}(V)$.	

\end{rem}

\begin{proof}[Proof of \autoref{pr:extBTD-def}.]
	By transferring the argumentation in \cite{LRRS}*{Sec.~4} to our setting, we first prove the positive case as follows.
	
	We may define $S$ local. Therefore, we can shrink $Y$ such that $\gamma_{+}$ is defined on $X=\pi^{-1}(Y)$.
	Set $\vartheta:=\theta+\gamma_+\geq 0$ and $\vartheta_\kappa=\theta_\kappa+\gamma_+\geq 0$.
	Let $A$ denote the analytic set of codimension $k+p-m$ such that $\pi(L(\vartheta))=\pi(L(\theta))\subset A$.
	On $X\minus\pi^{-1}(A)$, Bedford-Taylor (see \autoref{thm:BT-continuity}) implies that we can define $\vartheta^k\wedge T:= \lim_{\kappa\rightarrow\infty} \vartheta^{k} \wedge T$.
	The direct image $S':=\pi_\ast \big( \vartheta^{k} \wedge T \big)$ is a closed positive current on $Y\minus A$ which is the weak limit of the closed positive currents $S_\kappa:=\pi_\ast \big(\vartheta_\kappa^{k} \wedge T\big)$.
	We set $s:=\dim Y-k-p+m=n-k-p$.
	As $A$ is of codimension $k+p-m$, there is an $(s,s)$-bump form $\beta$ at $y$ with arbitrarily small support such that $A\cap \supp dd^c\beta=\emptyset$, see \autoref{rk:BumpFormExists}.
	Let $\alpha$ be a closed real $(1,1)$-form in the same \BC-class of $\vartheta$ (and $\vartheta_\kappa$).
	Since $\vartheta_\kappa=\theta_\kappa+\gamma_+$ is decreasing pointwise to $\vartheta=\theta+\gamma_+$, there is a sequence of $\alpha$-psh functions $q_\kappa$ with $dd^c q_\kappa=\vartheta_\kappa-\alpha$ which is decreasing pointwise to an $\alpha$-psh function $q$ with $dd^c q=\vartheta-\alpha$. We get
	\begin{align*}
		\int S_\kappa\wedge \beta
		=&\int \vartheta_\kappa^{k} \wedge T \wedge\pi^\ast \beta\\
		=&\int \Big(\alpha^k+ \sum\big._{i=1}^k\smbinom{k}{i}\alpha^{k-i}\wedge (dd^cq_\kappa)^{i}\Big)
			\wedge T \wedge\pi^\ast \beta\\
		=&\int \Big(\alpha^k\wedge\pi^\ast \beta+ q_\kappa\cdot\pi^\ast dd^c\beta\wedge\sum\big._{i=1}^k\smbinom{k}{i}\alpha^{k-i}\wedge (dd^cq_\kappa)^{i-1}\Big)\wedge T \\
		& \overset{\kappa\rightarrow\infty}{\longrightarrow}
		\int \Big(\alpha^k\wedge\pi^\ast \beta+ q\cdot\pi^\ast dd^c\beta\wedge\sum\big._{i=1}^k\smbinom{k}{i}\alpha^{k-i}\wedge (dd^cq)^{i-1}\Big)\wedge T
		=:S_\beta
	\end{align*}%
	whereby the convergence follows from Bedford-Taylor as $dd^c\beta$ has no support in $A$, see \autoref{thm:BT-continuity}.
	By \autoref{pr:convergence-of-positive-currents},
	we get that $S'$ uniquely extends to a closed positive current $S$ which is the weak limit of $S_\kappa$.
	
	\smallskip
	
	The quasipositive case is now a direct consequence of the following equation, \cf \eqref{eq:BTD-for-quasi}.
	\begin{equation}\begin{aligned} %
		\label{eq:def-exBTD-quasi}
		\pi_\ast\big([\theta]^k\wedge T\big) &
		= \pi_\ast\big([\theta+\gamma_+ -\gamma_+]^k \wedge T\big)
		:=\sum\big._{i=0}^k (-1)^i \smbinom{k}{i} \pi_\ast \left( [\theta+\gamma_+]^{k-i}\wedge\gamma_+^{i}\wedge T\right)\\
		& \overset{\Def}= \lim_{\kappa\rightarrow\infty} \sum\big._{i=0}^k (-1)^i \smbinom{k}{i} \pi_\ast\left( (\theta_\kappa+\gamma_+)^{k-i}\wedge\gamma_+^{i}\wedge T\right) = \lim_{\kappa\rightarrow\infty} \pi_\ast\big(\theta_\kappa^k\wedge T\big).
	\end{aligned}\end{equation}%
	Thereby, the terms %
	in the second line of \eqref{eq:def-exBTD-quasi} are well-defined by the first part of this proof.
	Due to the last equation in \eqref{eq:def-exBTD-quasi}, this definition is independent of the choice of $\gamma_+$.
\end{proof}

\begin{rem}\label{rk:extBTDsimplCont}
	If $\theta_\kappa$ is a sequence of $\gamma$-pos currents (not necessarily smooth) which is decreasing pointwise to $\theta$, then $L(\theta_\kappa)\subset L(\theta)$. Therefore, 
	$\pi_\ast\big( [\theta_\kappa]^{k} \wedge T\big)$ are also defined by \autoref{pr:extBTD-def} for all $\kappa$.
	Moreover,	the proof above works also for non-smooth $\theta_\kappa$ such that $\pi_\ast \big( [\theta]^{k} \wedge T \big)= \lim_{\kappa\rightarrow\infty} \pi_\ast \big( [\theta_\kappa]^{k} \wedge T\big)$. %
	In other words, the conclusions of \autoref{pr:extBTD-def} are also correct under
	the exactly same assumptions within the proposition beside of that $\theta_\kappa$ are allowed to be currents instead of forms.
	This is worth mentioning since we do not need the extra assumption on $\{\theta\}_{\BC}$ which will be assumed to obtain the (full) monotone continuity result, \autoref{thm:extBTD-continuity} below.
\end{rem}

\begin{rem}\label{rk:calculations-relative-positive}
	If there is a closed (positive) $(1,1)$-form $\omega$ on $Y$ such that  $\theta$ is $\pi^\ast\omega$-positive, then we obtain as in \eqref{eq:def-exBTD-quasi} that
	\begin{align*}%
		\pi_\ast\big([\theta]^k\wedge T\big)
		=&\sum\big._{i=0}^k (-1)^i \smbinom{k}{i} \pi_\ast \left( [\theta+\pi^\ast\omega]^{k-i}\wedge\pi^\ast\omega^{i}\wedge T\right)\\
		=&\sum\big._{i=0}^{k-(\dimF-p)} (-1)^i \smbinom{k}{i} \pi_\ast \left( [\theta+\pi^\ast\omega]^{k-i}\wedge T\right)\wedge\omega^{i}
	\end{align*}%
	where the last equation follows from the projection formula and that $\pi$'s fibres are of dimension \dimF. %
	For $k+p=\dimF+1$, we get that
	\begin{equation*}
		\pi_\ast\big([\theta]^{\dimF-p+1}\wedge T\big) = \pi_\ast \left( [\theta+\pi^\ast\omega]^{\dimF-p+1}\wedge T\right)-(\dimF-p+1)\omega\wedge\pi_\ast\left( [\theta+\pi^\ast\omega]^{\dimF-p}\wedge T\right).
	\end{equation*}
	If $T=\pi^\ast S$ for a closed positive current $S$ on $Y$, then %
	\begin{align*}
		\pi_\ast\big([\theta]^k\big)\wedge S
		=\ & \sum\big._{i=0}^{k-m} (-1)^i \smbinom{k}{i} \pi_\ast \big( [\theta+\pi^\ast\omega]^{k-i}\big)\wedge\omega^{i}\wedge S \hbox{\quad and}\\
		\pi_\ast\big([\theta]^{m+1}\big)\wedge S
		=\ & \pi_\ast \big( [\theta+\pi^\ast\omega]^{m+1}\big)\wedge S
		 -(m{+}1)\,\omega\wedge\pi_\ast\big( [\theta+\pi^\ast\omega]^{m}\big)\wedge S.
	\end{align*}%
\end{rem}

\bigskip

We get the following monotone continuity result for the extended \BTD MA products defined by \autoref{pr:extBTD-def}.
\begin{thm}\label{thm:extBTD-continuity}
	Let $\pi\colon X\rightarrow Y$ be a proper holomorphic submersion between complex manifolds $X$ and $Y$ with \dimF-dimensional fibres, %
	and let $T$ be a closed positive $(p,p)$-current on $X$.
	For $i=1\kdots t$,
	let $\theta_i$ be closed $\gamma$-pos $(1,1)$-currents on $X$
	for a (pos) $(1,1)$-form $\gamma$
	such that for any small enough open $V\subset Y$,
	there are closed positive $(1,1)$-forms $\gamma_{+}$ and $\alpha_{i,\pm}$ with $\gamma\leq\gamma_{+}$ and $\alpha_{i,+}-\alpha_{i,-}\in\{\theta_i\}_{\BC}$ on $\pi^{-1}(V)$  (for example, $\pi$ is Kähler).
	Let $\theta_{i,\kappa}\in\{\theta_i\}_{\BC}$ be closed $\gamma$-positive $(1,1)$-currents decreasing pointwise to $\theta_i$,
	\ie there are local $\gamma$-psh potentials of $\theta_{i,\kappa}$ decreasing pointwise to a $\gamma$-psh potential of $\theta_{i}$; %
	and let $k_i$ be positive integers.
	If $\bigcup_{i=1}^t \pi(L(\theta_i))$ is contained in an analytic $A$ with $\codim A\geq k_1+\ldots+k_t-m+p$, then
	\begin{equation}
		\pi_\ast \big( [\theta_t]^{k_t}\wedge \cdots \wedge[\theta_1]^{k_1}\wedge T\big)
		= \lim_{\kappa\rightarrow\infty} \pi_\ast \big([\theta_{t,\kappa}]^{k_t}\wedge \cdots \wedge[\theta_{1,\kappa}]^{k_1}\wedge T\big).
	\end{equation}
\end{thm}

\begin{proof}
	Without loss of generality, we may assume that $k_1=\ldots=k_t=1$. The assumption on the codimension of $A$ reads then $\codim A\geq t+p-\dimF$.
	By a recursive application of \autoref{pr:extBTD-def},
	we obtain that $\pi_\ast \big( [\theta_t]\wedge \cdots \wedge[\theta_1]\wedge T\big)$ and (in case that $\theta_{i,\kappa}$ is not smooth) $\pi_\ast \big([\theta_{t,\kappa}]\wedge \cdots \wedge[\theta_{1,\kappa}]\wedge T\big)$ are well-defined currents for all $\kappa$.
	
	As the statement is local with respect to $Y$, we may shrink $Y$ and $X=\pi^{-1}(Y)$ such that $\gamma_{+}$ and $\alpha_\pm$ are defined on $X$.
	Let us set $\vtheta_i:=\theta_{i}+\gamma_{+}$ and $\vtheta_{i,\kappa}:=\theta_{i,\kappa}+\gamma_{+}$. By assumption, these are closed positive currents, and we get that the sequence $\vtheta_{i,\kappa}$ is decreasing pointwise to $\vtheta_{i}$ in the sense defined above.
	To infer the statement in the quasipos case from the positive one, 
	we are going to prove
	\begin{equation}\label{eq:mixed-continuity-IA}
		\begin{split}
		&\pi_\ast \big( [\theta_t]\wedge \cdots \wedge[\theta_{i+1}]\wedge[\vtheta_{i}]\wedge \cdots \wedge[\vtheta_1]\wedge T\big)\\
		&\hspace*{10em}= \lim_{\kappa\rightarrow\infty} \pi_\ast \big([\theta_{t,\kappa}]\wedge \cdots \wedge[\theta_{i+1,\kappa}]\wedge[\vtheta_{i,\kappa}]\wedge \cdots \wedge[\vtheta_{1,\kappa}]\wedge T\big)
		\end{split}
	\end{equation}
	by induction over the number of $\theta_j$-factors, $t{-}i$.

	First, we proof the induction step, $t{-}i\mapsto t{-}i{+}1$ for $1\leq i\leq t$ (\ie $i\mapsto i{-}1$).
	By induction assumption, we have \eqref{eq:mixed-continuity-IA}.
	Yet, the induction assumption also implies that
	\begin{equation}
		\label{eq:mixed-continuity-IA+}
		\begin{split}
			&\pi_\ast \big( [\theta_t]\wedge \cdots \wedge[\theta_{i+1}]\wedge\gamma_{+}\wedge[\vtheta_{i-1}]\wedge \cdots \wedge[\vtheta_1]\wedge T\big)
			=\\
			&\hspace{10em}\lim_{\kappa\rightarrow\infty} \pi_\ast \big([\theta_{t,\kappa}]\wedge \cdots \wedge[\theta_{i+1,\kappa}]\wedge\gamma_{+}\wedge[\vtheta_{i,\kappa}]\wedge \cdots \wedge[\vtheta_{1,\kappa}]\wedge T\big)	
		\end{split}
	\end{equation}
	as there are $t{-}1-(i{-}1)=t{-}i$ of $\theta_j$-factors. Thereby, we use that $\gamma_+$ is positive and closed. 
	Summing up \eqref{eq:mixed-continuity-IA} and \eqref{eq:mixed-continuity-IA+}, we obtain 
	(since $\vtheta_i=\theta_i+\gamma_{+}$ and $\vtheta_{i,\kappa}=\theta_{i,\kappa}+\gamma_+$)
	\pagebreak[0]
	\begin{equation*}	%
		\pi_\ast \big( [\theta_t]\wedge \cdots \wedge[\theta_{i}]\wedge[\vtheta_{i-1}]\wedge \cdots \wedge[\vtheta_1]\wedge T\big)
		=\!\lim_{\kappa\rightarrow\infty} \pi_\ast \big([\theta_{t,\kappa}]\wedge \cdots \wedge[\theta_{i,\kappa}]\wedge[\vtheta_{i-1,\kappa}]\wedge \cdots \wedge[\vtheta_{1,\kappa}]\wedge T\big)
	\end{equation*}
	which is the case $t-(i-1)=(t{-}i)+1$. This concludes the induction step.
	\smallskip

	Next, we prove the base case $t{-}i=0$, \ie there is not any $\theta_j$-factor and all factors are positive.
	We obtain that the currents $S=\pi_\ast \big( [\vtheta_t]\wedge \cdots \wedge[\vtheta_1]\wedge T\big)$ and $S_\kappa:=\pi_\ast \big( [\vtheta_{t,\kappa}]\wedge \cdots \wedge[\vtheta_{1,\kappa}]\wedge T\big)$ are closed positive $(t{+}p{-}m,t{+}p{-}m)$-currents. Therefore, the weak convergence of $S_\kappa$ to $S$ follows from \autoref{pr:convergence-of-positive-currents} when its assumptions (i) and (ii) are satisfied which we show as follows.
	\smallskip

	\ManualItem{i}
	Since $\vtheta_{i,\kappa}$ is decreasing pointwise to $\vtheta_i$, we get that $L(\vtheta_{i,\kappa})\subset L(\vtheta_i)$ for all $i$ and $\kappa$.
	By \autoref{thm:BT-continuity}, we obtain that 
	\begin{equation*}
		\vtheta_t\wedge \cdots \wedge\vtheta_1\wedge T
		= \lim_{\kappa\rightarrow\infty}\vtheta_{t,\kappa}\wedge \cdots \wedge\vtheta_{1,\kappa}\wedge T
	\end{equation*}
	on $X\minus \pi^{-1}(A)\subset X\minus \bigcup_i \pi(L(\vtheta_i))$, and 
	\begin{equation*}
		S=\pi_\ast \big( \vtheta_t\wedge \cdots \wedge\vtheta_1\wedge T\big)
		= \lim_{\kappa\rightarrow\infty} \pi_\ast \big(\vtheta_{t,\kappa}\wedge \cdots \wedge\vtheta_{1,\kappa}\wedge T\big)=\lim_{\kappa\rightarrow\infty}S_\kappa
	\end{equation*}
	on $Y\minus A$.
	\smallskip

	\ManualItem{ii}
	Since $\codim A\geq s:=t+p-m$, for each point $x\in A$, there is an $(n{-}s,n{-}s)$-bump form $\beta$ %
	with arbitrarily small support and $dd^c\beta\cap A=\emptyset$, see \autoref{rk:BumpFormExists}.
	By assumption, $\alpha:=\alpha_{t,+}+\gamma_+-\alpha_{t,-}$ is the difference of two closed positive forms and in the same \BC-class as $\vtheta_t=\theta_t+\gamma_+$, \ie there is a quasipsh function $q$ on $X$ such that $\vtheta_t-\alpha=dd^c q$. By the assumption on $\vtheta_{t,\kappa}$, we get quasipsh functions $q_\kappa$ on $X$ decreasing pointwise to $q$ \st $\vtheta_{t,\kappa}-\alpha=dd^c q_\kappa$.
	We get %
	\begin{align*}%
		\int	\pi_\ast \big( [\vtheta_t]\wedge \cdots \wedge[\vtheta_1]\wedge T\big)\wedge \beta
		= & \int \pi_\ast \big([dd^c q]\wedge[\vtheta_{t-1}]\wedge \cdots \wedge[\vtheta_1]\wedge T\big)\wedge \beta\\
		& + \int \pi_\ast \big([\vtheta_{t-1}]\wedge \cdots \wedge[\vtheta_1]\wedge \alpha\wedge T\big)\wedge \beta%
	\end{align*}%
	where the last term is defined since $\alpha$ is the difference of two closed positive forms.
	Let $q^{(\lambda)}$%
	be smooth $\alpha$-psh functions decreasing pointwise to $q$, and let $\vtheta_i^{(\lambda)}$ be positive forms decreasing pointwise to $\vtheta_i$ for $i=1\kdots t{-}1$.
	Since recursively defined by \autoref{pr:extBTD-def}, we obtain %
	\begin{align*}%
		&\int \pi_\ast \big([dd^c q]\wedge[\vtheta_{t-1}]\wedge \cdots \wedge[\vtheta_1]\wedge T\big)\wedge \beta
		\\&\hspace*{10em}
		=
		\lim_{\lambda_t\rightarrow \infty }{\cdots}\lim_{\lambda_1\rightarrow \infty }
		\int dd^c q^{(\lambda_{t})}\wedge\vtheta_{t-1}^{(\lambda_{t-1})}\wedge \cdots \wedge\vtheta_1^{(\lambda_{1})}\wedge T\wedge \pi^\ast \beta.%
	\end{align*}%
	Since $q^{(\lambda)}$ are (globally defined) functions on $X$, we can apply integration by parts (as the definition of $d$ and $d^c$ on currents)
	in the limit-term of the RHS and get 
	\begin{align*}
		\int	dd^c q^{(\lambda_{t})}\wedge\vtheta_{t-1}^{(\lambda_{t-1})}\wedge \cdots \wedge\vtheta_1^{(\lambda_{1})}\wedge T\wedge \pi^\ast\beta%
		\hspace{-5em}&\\
		& = 	\int	q^{(\lambda_{t})}\cdot\vtheta_{t-1}^{(\lambda_{t-1})}\wedge \cdots \wedge\vtheta_1^{(\lambda_{1})}\wedge T\wedge dd^c\pi^\ast \beta\\
		& = \int_{X\minus U} q^{(\lambda_{t})}\cdot\vtheta_{t-1}^{(\lambda_{t-1})}\wedge \cdots \wedge\vtheta_1^{(\lambda_{1})}\wedge T\wedge dd^c\pi^\ast \tilde\beta\\
		& = \int_{X\minus U} dd^c q^{(\lambda_{t})}\wedge\vtheta_{t-1}^{(\lambda_{t-1})}\wedge \cdots \wedge\vtheta_1^{(\lambda_{1})}\wedge T\wedge \pi^\ast\tilde\beta\\
		&\hspace*{-.35em}
		\overset{\lambda_1\rightarrow\infty }{\longrightarrow}\cdots
		\overset{\lambda_t\rightarrow\infty }{\longrightarrow}
		 \int_{X\minus U}	\hspace{-1em}dd^c q\wedge \vtheta_{t-1}\wedge \cdots \wedge\vtheta_1\wedge T\wedge \pi^\ast\tilde\beta
	\end{align*}	%
	where $U=U_\delta=\pi^{-1}(V_\delta)$ for a neighbourhood $V=V_\delta$ of $A$ in $X$ (so small that
	\pagebreak[0]
	$V_\delta\cap \supp dd^c\beta=\emptyset$)
	and $\tilde\beta=\chi\beta$ for a smooth cutoff function $\chi$ with support on $Y\minus U$ and 
	$\chi|_{\supp dd^c\beta}=1$. %
	As $L(\vartheta_i)=L(\theta_i)\subset U$, the convergence is implied by Bedford-Taylor, see \autoref{thm:BT-continuity}.
	Since $dd^cq$ equals $\vtheta_t-\alpha$, we get
	\begin{equation}\begin{split}
		\label{eq:base-case-split}
		\int	\pi_\ast \big( [\vtheta_t]\wedge \cdots \wedge[\vtheta_1]\wedge T\big)\wedge \beta
		=&  	\int_{X\minus U}	\hspace{-1em}\vtheta_t\wedge \cdots \wedge\vtheta_1\wedge T\wedge \pi^\ast\tilde\beta\\
		&+	\int_{U}	\pi_\ast \big( [\vtheta_{t-1}]\wedge \cdots \wedge[\vtheta_1]\wedge\alpha\wedge T\big)\wedge \beta.
	\end{split}\end{equation}
	Repeating exactly the same arguments in the lines above replacing $\vtheta_i$ by $\vtheta_{i,\kappa}$ and $q$ by $q_\kappa$, we obtain the same formula for all $\kappa$:
	\begin{equation}\label{eq:base-case-split+}
	\begin{split}
		\int	\pi_\ast \big( [\vtheta_{t,\kappa}]\wedge \cdots \wedge[\vtheta_{1,\kappa}]\wedge T\big)\wedge \beta
		=&  	\int_{X\minus U}	\hspace{-1em} \vtheta_{t,\kappa}\wedge \cdots \wedge\vtheta_{1,\kappa}\wedge T\wedge \pi^\ast\tilde\beta\\
		&+	\int_{U} \pi_\ast \big( [\vtheta_{t-1,\kappa}]\wedge \cdots \wedge[\vtheta_{1,\kappa}]\wedge\alpha\wedge T\big)\wedge \beta.
	\end{split}
	\end{equation}

	By \autoref{thm:BT-continuity}, we get that the first term on the RHS in \eqref{eq:base-case-split+} converges to the first one on the RHS in \eqref{eq:base-case-split}.
	We obtain the same for the second term  %
	as we may assume that the proposition is already proven for $(t{-}1)$-factors (\ie by induction over $t$) since $\alpha$ is the difference of two closed positive forms.
	Hence, we get
	\begin{equation*}
		\int	\pi_\ast \big( [\vtheta_t]\wedge \cdots \wedge[\vtheta_1]\wedge T\big)\wedge \beta
		= \lim_{\kappa\rightarrow\infty}	\int	\pi_\ast \big( [\vtheta_{t,\kappa}]\wedge \cdots \wedge[\vtheta_{1,\kappa}]\wedge T\big)\wedge \beta.
	\end{equation*}
	\EoPfwEq
\end{proof}

Using the idea of \cite{LRSW} to define the wedge product of pushforwards currents with the fibre product,  %
\autoref{thm:extBTD-continuity} implies a monotone continuity result for the recursively defined mixed products:

\begin{cor}\label{cr:extBTD-continuity-special}
	For $i=1\kdots t$,
	let $\pi_i\colon X_i\rightarrow Y$ be proper holomorphic submersions between complex manifolds $X_i$ and $Y$ with $m_i$-dimensional fibres,
	let $\theta_i$ be closed $\gamma_i$-positive $(1,1)$-currents on $X_i$ for (positive) $(1,1)$-forms $\gamma_i$ on $X_i$
	such that for any small enough open $V\subset Y$,
	there are \emph{closed} positive $(1,1)$-forms $\gamma_{i,+}$ and $\alpha_{i,\pm}$ on $\pi_i^{-1}(V)$
	with $\gamma_i\leq\gamma_{i,+}$ and $\alpha_{i,+}-\alpha_{i,-}\in\{\theta_i\}_{\BC}$ 
	(for example, $\pi_i$ are Kähler).
	Let $\theta_{i,\kappa}\in\{\theta_i\}_{\BC}$ be closed $\gamma_i$-positive $(1,1)$-currents decreasing pointwise to $\theta_i$,
	let $T$ be a closed pos $(p,p)$-current on $X_1$, and let $k_i$ be positive integers.

	If $\bigcup_{i=1}^t \pi_i(L(\theta_i))$ is contained in an analytic $A$ with $\codim A\geq k_1+\ldots+k_t-m_1-\ldots-m_t+p$, then
	\begin{equation}
	\label{eq:one-step-limit}
		\pi_{t,\ast} \big( [\theta_t]^{k_t}\big)\wedge \cdots \wedge\pi_{1,\ast}\big([\theta_1]^{k_1}\wedge T\big)
		= \lim_{\kappa\rightarrow\infty} \pi_{t,\ast} \big([\theta_{t,\kappa}]^{k_t}\big)\wedge \cdots \wedge\pi_{1,\ast} \big([\theta_{1,\kappa}]^{k_1}\wedge T\big).
	\end{equation}
\end{cor}
\begin{proof}
	Let $\varpi\colon\tilde X\rightarrow Y$ be the fibre product of $\tilde X=X_t\times_Y\cdots\times_YX_1$,
	let $\pr_i\colon \tilde X\rightarrow X_i$ be its projection on the $i$th component, set $\tilde\theta_i:=\pr_i^\ast\theta_i$, set $\tilde\theta_{i,\kappa}:=\pr_i^\ast\theta_{i,\kappa}$, and set $\gamma:=\sum_i\pr_i^\ast\gamma_i$ (such that $\gamma\leq\sum_i\pr_i^\ast\gamma_{i,+}$ on $\varpi^{-1}(V)$).
	Let $\theta_i^{(\lambda)}$ be $\gamma_i$-pos $(1,1)$-forms decreasing pointwise to $\theta_i$, and set $\tilde\theta_i^{(\lambda)}:=\pr_i^\ast \theta_i^{(\lambda)}$.
	As recursively defined by \autoref{pr:extBTD-def}
	(using the projection formula, as well),  we get %
	\begin{align}%
		\nonumber
		\pi_{t,\ast} \big( [\theta_t]^{k_t}\big)\wedge \cdots \wedge\pi_{1,\ast}\big([\theta_1]^{k_1}\wedge T\big)
		\hspace{-5em}&\\
		\label{eq:fibre-power-for-extBTD}
		&\overset{\Def}{=}
		\lim_{\lambda_t\rightarrow 0 }\cdots\lim_{\lambda_1\rightarrow 0 }
		\pi_{t,\ast} \big( (\theta_t^{(\lambda_t)})^{k_t}\big)\wedge \cdots \wedge\pi_{1,\ast}\big((\theta_1^{(\lambda_1)})^{k_1}\wedge T\big)\\
		\nonumber		
		&\,=
		\lim_{\lambda_t\rightarrow 0 }\cdots\lim_{\lambda_1\rightarrow 0 }
		\varpi_\ast \big( (\tilde\theta_{t}^{(\lambda_t)})^{k_t}\wedge \cdots \wedge(\tilde\theta_{1}^{(\lambda_1)})^{k_1}\wedge\pr_1^\ast T\big)\\
		&\overset{\Def}{=}
		\varpi_\ast \big( [\tilde\theta_{t}]^{k_t}\wedge \cdots \wedge[\tilde\theta_{1}]^{k_1}\wedge\pr_1^\ast T\big)
		\nonumber
	\end{align}%
	where the second equation follows from \autoref{lm:fibre-products-and-direct-images}.
	Repeating these arguments for $\theta_{i,\kappa}$ instead of $\theta_i$, we obtain also 
	\begin{equation}\label{eq:fibre-power-for-extBTD-eps}
		\pi_{t,\ast} \big( [\theta_{t,\kappa}]^{k_t}\big)\wedge \cdots \wedge\pi_{1,\ast}\big([\theta_{1,\kappa}]^{k_1}\wedge T\big)
		= \varpi_\ast \big( [\tilde\theta_{t,\kappa}]^{k_t}\wedge \cdots \wedge[\tilde\theta_{1,\kappa}]^{k_1}\wedge\pr_1^\ast T\big)
	\end{equation}
	for all $\kappa$. By \autoref{thm:extBTD-continuity}, the RHS of \eqref{eq:fibre-power-for-extBTD-eps} converges weakly to the right side of \eqref{eq:fibre-power-for-extBTD} as $\kappa\rightarrow\infty$. Hence, the same applies to the LHS of the equations what proves the claimed. %
\end{proof}

\section{Currents with analytic singularities}
\label{sc:currents-analy-sing}
\noindent
Throughout this section, let $X$ be a complex manifold of dimension \dimX. 
The main objective of this section is the study of wedge products of currents with analytic singularities. %

\begin{df}\label{df:QPshAS}
	We call $q$ a \emph{quasipsh function with analytic singularities} in $Z$ if
	locally $q=u+\smth$ whereby $u=c\log\| F\|^2+b$ is a psh function with analytic singularities in the (reduced) analytic set $\{F=0\}$ for a positive constant $c$, a tuple $F$ of holomorphic functions and bounded $b$.
	Moreover, we call $q$ a quasipsh function with \emph{neat} analytic singularities %
	if $b$ can be chosen to be smooth, \ie locally $q=c\log\| F\|^2+\smth$.

	We say that a closed positive $\theta$ has (\emph{neat}) \emph{analytic singularities}  when its local \ddc-potential is psh with (neat) analytic singularities.
	$\theta$ is called \emph{quasipositive} (shortly \emph{quasipos}) with (neat) analytic singularities %
	if it is locally the sum of a closed positive current with (neat) analytic singularities and a closed real $(1,1)$-form.
	This is equivalent to that locally $\theta=dd^c q$ for a quasipsh function $q$ with (neat) analytic singularities. %
\end{df}

Let $u_1\kdots u_k$ be psh functions with analytic singularities in $Z_1\kdots Z_k$, resp.
For constructible sets $W_2\kdots W_k$ such that $W_j\subset Z_j^c:=X\minus Z_j$, we can define the MA product
\begin{equation}\label{eq:general-AW-product}
	dd^c u_k\wedge\one_{W_k} dd^c u_{k-1}\wedge \one_{W_{k-1}}\cdots dd^c u_2\wedge \one_{W_2} dd^c u_1
\end{equation}
recursively (from right to left) as a closed positive current, see \cite{LRSW}*{Prop.~3.2}.
In case of $W_j=X\minus Z_j$, we call \eqref{eq:general-AW-product} the (mixed) \emph{Andersson-Wulcan (AW) MA product} of $dd^c u_k\kdots dd^c u_1$.
For a psh function $u$ with analytic singularities in $Z$, $(dd^c u)_\AW^k$ denotes the $k$-power $dd^c u\wedge\one_{Z^c} \cdots\one_{Z^c} dd^c u$ which was introduced in \cite{AW14-MA}.

Please note there is no restriction on $k$. Moreover, this definition of (mixed) Monge-Amp\`ere products is an extension of %
\BTD MA, see \cite{LRSW}*{Prop.~3.4}, %
and (mixed) products of Bochner-Martinelli currents $dd^c\log\Vert F_k\Vert^2$ for holomorphic tuples $F_k$ as it was introduced in \cite[Sec.~5]{ASWY17}. %
Also, we would like to point out that \eqref{eq:general-AW-product}  %
can be defined for quasipsh potentials, as well, see \cites{Blocki19-quasiMA,LSW} and also \cite{LRSW}*{Lem.~3.5}
(\ip for $\gamma$-psh potentials for a (pos) form  $\gamma$).

\bigskip

As in \cite{LRSW}*{Sec.~3} and \cite{LSW}*{Sec.~2}, we consider the following class of currents.

\begin{df}\label{df:currents-w-analy-sing}
We call a (closed) real $(p,p)$-current $T$ a (closed) \emph{current with analytic singularities} if it is locally of the following form:
\begin{equation}\label{eq:local-form-of-currents-with-as}
	\sum{\big.\!}_i\, \alpha_i\wedge \one_{W_{i,k_i+1}}dd^c u_{i,k_i}\wedge  \one_{W_{i,k_i}}dd^c u_{i,k_i-1}\one_{W_{i,k_i-1}} \cdots\wedge \one_{W_{i,2}}dd^c u_{i,1} \one_{W_{i,1}}
\end{equation}
with smooth closed real forms $\alpha_i$ of degree $(p{-}k_i,p{-}k_i)$, psh $u_{i,j}$ with analytic singularities, and constructible sets $W_{i,j}$ such that $L(u_{i,j})\subset W_{i,j}^c$ for $j\leq 2$.
As $\one_{\{\pt\}}T=0$ for $(p,p)$-currents with $p<n$, we can omit the last condition on $W_{i,j}$ if $L(u_{i,j})$ is a set of isolated points.

If furthermore $L(u_{i,1})\subset W_{i,1}^c$, $T\wedge\cdot$ denotes an operator  on closed currents with analytic singularities %
(which is given recursively and well-defined due to the condition on $W_{i,1}$). %
\end{df}
In particular, the AW MA product of closed quasipositive $(1,1)$-currents with analytic singularities (see \eqref{eq:general-AW-product}) is a closed current with analytic singularities. Let us use the following more specific notation for such currents. %
\begin{df}\label{df:qpos-currents-w-analy-sing}
	A closed $(p,p)$-current $T$ with analytic singularities is called \emph{quasipositive} (shortly \emph{quasipos}) if $T$ is locally of the following form:
	\begin{equation}\label{eq:local-form-of-qpos-currents-with-as}
		\sum{\big.\!}_i\, c_i\one_{W_{i,p+1}}dd^c q_{i,p}\wedge  \one_{W_{i,p}}dd^c q_{i,p-1}\one_{W_{i,p-1}} \cdots\wedge \one_{W_{i,2}}dd^c q_{i,1} \one_{W_{i,1}}
	\end{equation}
	with positive constants $c_i$, quasipsh $q_{i,j}$ with analytic singularities, and constructible sets $W_{i,j}$ such that $L(q_{i,j})\subset W_{i,j}^c$ for $j\geq 2$.

	Obviously, a closed $(p,p)$-current $T$ with analytic singularities is quasipos if and only if it is locally of the form as in \eqref{eq:local-form-of-currents-with-as} with $\alpha_i$ is either a positive constant or the wedge product of closed real $(1,1)$-forms.
	In particular, locally $T$ is the difference of two closed positive currents.
\end{df}

\begin{rem}\label{rk:PFofQposCurAS}
	Let $\pi\colon X\rightarrow Y$ be a proper holomorphic submersion, and
	let $\gamma$ be a (pos) $(1,1)$-form such that for all small enough open $V\subset Y$, there is a \emph{closed} positive $(1,1)$-form $\gamma_+$ with $\gamma_+\geq\gamma$ on $U=\pi^{-1}(V)$, \eg $\pi$ is Kähler, see \autoref{rk:Kaehler-pi}.
	Let $\theta_1\kdots\theta_p$ be closed $\gamma$-pos $(1,1)$ currents with analytic singularities,
	and let $W_2\kdots W_{p}$ be constructible sets with $L(\theta_j)\subset W_{j}^c$ for $j\geq 2$.
	Then, %
	$T=\theta_{p}\wedge  \one_{W_{p}}\cdots \theta_2\wedge \one_{W_{2}}\theta_{1}$ %
	can be decomposed into the difference $T_+-T_-$ of two closed positive currents $T_\pm$ with analytic singularities
	on $U=\pi^{-1}(V)$ for all small enough $V$ (to be more precise, on all $U$ on which the $\gamma_+$ from above exists).

	Let us prove this by induction over $p$.
	The base case $p=1$ is trivial.
	For the induction step, we may assume that there are two closed positive currents $S_\pm$ with analytic singularities such that
	$\theta_{p-1}\wedge  \one_{W_{p-1}}\cdots \theta_2\wedge \one_{W_{2}}\theta_{1}=S_+-S_-$.
	Then,
	\begin{equation*}
		T\!=\!\theta_{p}\wedge  \one_{W_{p}}(S_+-S_-)
		\!=\!\big((\theta_{p}{+}\gamma_+)\wedge \one_{W_{p}}S_+ + \gamma_+\wedge \one_{W_{p}}S_-\big) 
		- \big((\theta_{p}{+}\gamma_+)\wedge \one_{W_{p}}S_- + \gamma_+\wedge \one_{W_{p}}S_+\big)
	\end{equation*}
	where minuend and subtrahend are closed and positive.	
\end{rem}

\medskip

We obtain two classes of currents which are closed under the following operations:
\\\ManualItem{1} wedge products with closed real forms (of bidegree $(1,1)$);
\\\ManualItem{2} $\one_W$ for any constructible set $W$; and
\\\ManualItem{3} $dd^c q \wedge \one_{Z^c}$ for any quasipsh function $q$ with analytic singularities in $Z$.
\\
If $T$ and $S$ are (quasipos) currents with analytic singularities, then $S\wedge T$ is well-defined and a (quasipos) current with analytic singularities where $S\wedge$ denotes the operator as defined in \autoref{df:currents-w-analy-sing}.

In this sense, the AW MA product has proven to be very useful, \eg we can define the Monge-Ampère operator for arbitrary degrees.
Yet, it lacks certain properties. For example, the product is neither commutative nor distributive, see \cite{LRSW}*{Exp.~3.1}.
Furthermore, although it coincides with the \BTD MA product whenever this defined (see \cite{LRSW}*{Prop.~3.4}), the pushforward of an AW MA product is in general different form the (natural) extension of \BTD defined by \autoref{pr:extBTD-def} as the following example shows; \cf \cite{LRSW}*{Exp.~8.3}.

\begin{exa}
	Let $X$ be $\CC_{z}{\times}\PP^1$, let $\pi$ be the projection on $Y=\CC_{z}$, %
	let $u$ denote the psh function $\log |z|^2$ on $X$, and let $\omega_\FS$ be the pullback of the Fubini-Study form on $\PP^1$.
	Then, $\pi_\ast\big((\omega_\FS{+}dd^c u)_\AW^{2}\big)= \pi_\ast\big(\omega_\FS\wedge [z{=}0]\big) = [z{=}0]$. %
	On the other side,
	if $u_\kappa$ is a sequence of smooth psh functions decreasing pointwise to $u$, then
	$(\omega_\FS{+}dd^c u_\kappa)^2= (dd^c u_\kappa)^2 + 2\omega_\FS\wedge dd^c u_\kappa+0$
	such that $\pi_\ast \big((\omega_\FS+dd^c u_\kappa)^2\big)=2\pi_\ast \big(\omega_\FS\wedge dd^c u_\kappa)$
	which converges to $2[z{=}0]$, \ie
	\begin{equation*}
		\pi_\ast\big([\omega_\FS{+}dd^c u]^{2}\big)=2[z{=}0]\neq \pi_\ast\big((\omega_\FS{+}dd^c u)_\AW^{2}\big).
	\end{equation*}
\end{exa}

Roughly speaking, the AW MA product removes too much. Inspired by \cite{ABW}*{Thm~1.2}, Lärkäng, Raufi, Wulcan and the author give the following definition in \cites{LRSW, LSW} which keeps track of the removed part by adding an extra term to the AW MA product. 

\begin{df}\label{def:genMA}
	For a closed quasipositive $(1,1)$-current $\theta$ with analytic singularities in $Z$, closed real $(1,1)$-form $\alpha$, and closed current $T$ with analytic singularities,
	we get that the so-called \emph{generalized Monge-Ampère product} of $\theta$ and $T$ defined by
	\begin{equation}\label{eq:defgMA}
		[\theta]_\alpha\wedge T := \theta\wedge\one_{Z^c} T+\alpha\wedge\one_{Z} T
	\end{equation}
	is a well-defined closed real current with analytic singularities which is quasipos when $T$ is.
	This is also defined and a (quasipos) current with analytic singularities for $\alpha$ replaced by a quasipos $(1,1)$-current $\eta$ with analytic singularities in isolated points due to \BTD.
\end{df}

Generalizing \cite{LRSW}*{Thm~1.2}, we obtain the following correlation between the extended \BTD MA product defined by \autoref{pr:extBTD-def} and pushforwards of the generalized MA products defined by  \eqref{eq:defgMA}.

\begin{thm}\label{thm:extBTD-coincides-genMA}
	Let $\pi\colon X\rightarrow Y$ be a proper holomorphic submersion between complex manifolds with \dimF-dimensional fibres, %
	let $\gamma$ be a (pos) $(1,1)$-form, 
	and let $T$ be a closed real $(p,p)$-current
	such that for any small enough $V\subset Y$,
	there are a closed positive form $\gamma_+$ with $\gamma\leq\gamma_+$ and closed positive currents $T_{\pm}$ with analytic singularities and $T=T_+-T_-$ on $\pi^{-1}(V)$ (\cf \autoref{rk:PFofQposCurAS}).
	Let $\theta$ be a $\gamma$-pos $(1,1)$-current with analytic singularities in $Z$,
	and let  $\eta\in\{\theta\}_{\BC}$ be a closed $\gamma$-pos $(1,1)$-current with neat analytic singularities in isolated points.
	Then, we get 
	\begin{equation*}
		\pi_\ast \big([\theta]^k\wedge T) = \pi_\ast \big([\theta]_\eta^k\wedge T)
	\end{equation*}
	for all $k+p\leq\codim \pi(Z)+m$.
	In particular, it is independent of $\eta$.
\end{thm}

Let us point out in contrast to the general case,
$S:=\pi_\ast\big([\theta]^k\wedge T)$ is actually given as a direct image of a current on $X$ for $\theta$ with analytic singularities. We will utilize this fact to estimate the Lelong numbers of $S$, see \autoref{cr:Lelong-estimate-pushforwards-AW-products}.

\begin{rem}
	By \autoref{rk:PFofQposCurAS}, we can apply \autoref{thm:extBTD-coincides-genMA} recursively. In particular, we get  %
	\begin{equation}\label{eq:extBTD-gMA-I}
		\pi_\ast \big([\theta_t]^{k_t}\wedge \cdots \wedge [\theta_1]^{k_1} \big)=\pi_\ast \big([\theta_t]_{\eta_t}^{k_t}\wedge \cdots \wedge [\theta_1]_{\eta_1}^{k_1} \big)
	\end{equation}
	for closed $\gamma$-pos $(1,1)$-currents $\theta_1\kdots\theta_t$ with analytic singularities in $Z_i$,
	closed $\gamma$-pos $(1,1)$-currents $\eta_1\kdots\eta_t$ with neat analytic singularities in isolated points such that $\eta_i\in\{\theta_i\}_{\BC}$ and 
	$k_1+\ldots+k_t\leq\codim \bigcup_{i=1}^t\pi(Z_i)+m$
	($\gamma$ as above).
\end{rem}

\begin{rem}\label{rk:gMAwFP}
	Using the idea from \cite{LRSW} (motivated by \autoref{lm:fibre-products-and-direct-images}),
	we can use the fibre product to define the wedge product of pushforwards of currents with analytic singularities in the following setting.
	For $i=1\kdots t$,
	let $\pi_i\colon X_i\rightarrow Y$ be proper holomorphic submersions with $m_i$-dimensional fibres, %
	let $\theta_i$ be closed $\gamma_i$-pos $(1,1)$-currents with analytic singularities in $Z_i$, %
	and let $\alpha_i\in\{\theta_i\}_{\BC}$ be closed $\gamma_i$-pos $(1,1)$-forms where 
	$\gamma_i$ are (pos) $(1,1)$-forms, each on $X_i$.
	Following \cite{LRSW}*{Thm~1.1}, we define
	\begin{equation}\label{eq:gMA-mixedPF}	%
		\pi_{t,\ast} \big( [\theta_t]_{\alpha_t}^{k_t}\big)\wedge \cdots \wedge\pi_{1,\ast}\big([\theta_1]_{\alpha_1}^{k_1}\big):=\varpi_\ast \big([\tilde\theta_t]_{\tilde\alpha_t}^{k_t}\wedge \cdots \wedge [\tilde\theta_1]_{\tilde\alpha_1}^{k_1} \big)
	\end{equation}
	where $\varpi\colon\tilde X\rightarrow Y$ is the fibre product, $\tilde X=X_t\times_Y\cdots\times_YX_1$, $\pr_i\colon \tilde X\rightarrow X_i$ denotes the projection on the $i$th component, $\tilde\theta_i:=\pr_i^\ast\theta_i$ and $\tilde\alpha_i:=\pr_i^\ast\alpha_i$.
	In \eqref{eq:gMA-mixedPF}, we cannot replace $\alpha_i$ by currents $\eta_i$ with neat analytic singularities in isolated points since $\pr_i^\ast\eta_i$ would have singularities in sets of higher codimension than points.
	If for any small enough open $V\subset Y$,
	there are {closed} positive $(1,1)$-forms $\gamma_{i,+}$ with $\gamma_i\leq\gamma_{i,+}$  on $\pi_i^{-1}(V)$, then \eqref{eq:extBTD-gMA-I} implies
	\begin{equation}	%
		\pi_{t,\ast} \big( [\theta_t]^{k_t}\big)\wedge \cdots \wedge\pi_{1,\ast}\big([\theta_1]^{k_1}\big)
		= 
		\pi_{t,\ast} \big( [\theta_t]_{\alpha_t}^{k_t}\big)\wedge \cdots \wedge\pi_{1,\ast}\big([\theta_1]_{\alpha_1}^{k_1}\big)
	\end{equation}	
	for all  $k_1+\ldots+k_t\leq \codim\bigcup_i{\pi(Z_i)}+m_1+\ldots+m_t$ (\cf proof of \autoref{cr:extBTD-continuity-special}). %
\end{rem}

\begin{proof}[Proof of \autoref{thm:extBTD-coincides-genMA}.]
	Let $\alpha$ be in the same \BC-class as $\theta$ (and $\eta$), \ie there exist a quasipsh $q$ with analytic singularities and a quasipsh $v$ with neat analytic singularities in isolated points such that $\theta=dd^cq+\alpha$ and $\eta=dd^cv+\alpha$.
	By the definition of the generalized MA product \eqref{eq:defgMA},
	we get $[\theta]_{\eta}\wedge T=[dd^cq+\alpha]_{dd^cv+\alpha}\wedge T=\big([dd^cq]_{dd^cv}+\alpha\big)\wedge T$.
	We set $q_\kappa:=\max\{q,v-\kappa\}$. %
	As locally, there is a constant $C$ \st $q+C\|z\|^2$ and $v+C\|z\|^2$ are psh (where $z$ denotes coordinates of $X$), we get $q_\kappa\overset{\loc}{=}\max\{q+C\|z\|^2,v-\kappa+C\|z\|^2\}-C\|z\|^2$ are quasipsh. %
	Moreover, $q_\kappa$ have neat analytic singularities %
	in isolated points and the sequence is decreasing pointwise to $q$, \ip $\theta_\kappa:=dd^cq_\kappa+\alpha\longrightarrow\theta$ weakly as $\kappa\rightarrow\infty$. %
	On any $U:=\pi^{-1}(V)$ for small enough open $V\subset Y$ where $T=T_+-T_-$ by the assumption,
	\autoref{lm:gMA-continuity-singular} below implies
	\begin{equation*}
		[\theta]_{\eta}\wedge T_\pm
		=\big([dd^cq]_{dd^cv}+\alpha\big)\wedge T_\pm
		=\lim_{\kappa\rightarrow\infty} (dd^c q_\kappa+\alpha)\wedge T_\pm
		=\lim_{\kappa\rightarrow\infty} \theta_\kappa\wedge T_\pm.
	\end{equation*}
	By taking the pushforwards of both sides of the equation, we get the claimed as in the proof of
	\autoref{pr:extBTD-def}, see \autoref{rk:extBTDsimplCont}. %
\end{proof}

\begin{lma} \label{lm:gMA-continuity}
	Let $T$ be a closed current with analytic singularities, and
	let $q$ be a quasipsh function with analytic singularities in $Z$.
	If $\rho_\kappa\colon \RR\rightarrow\RR$ is a sequence of non-decreasing convex and bounded from below functions such that $\rho_\kappa$ is decreasing pointwise to the identity as $\kappa\rightarrow\infty$, then
	\begin{equation}\label{eq:AWMA-continuity}
		dd^c \rho_\kappa\circ q  \wedge T
		\overset{\kappa\rightarrow\infty}\longrightarrow
		dd^c q \wedge \one_{Z^c}T.
	\end{equation}
	In particular, if %
	$v$ is a smooth %
	function, then %
	\begin{equation}\label{eq:gMA-continuity}
		[dd^c q]_{dd^c v}\wedge T = \lim_ {\kappa\rightarrow\infty} dd^c \big(\rho_\kappa\circ (q-v)+v\big)  \wedge T.
	\end{equation}
\end{lma}
\begin{proof}
	\eqref{eq:AWMA-continuity} is a direct consequence of the definition of the AW MA product, see \cite{LRSW}*{Prop.~3.2}. %
	Now, \eqref{eq:AWMA-continuity} implies
	\begin{equation*}
		dd^c \rho_\kappa\circ (q-v)  \wedge T
		\overset{\kappa\rightarrow\infty}\longrightarrow
		dd^c (q-v) \wedge \one_{Z^c}T
	\end{equation*}
	which is equivalent to \eqref{eq:gMA-continuity} since we can add $dd^cv\wedge T$ on both sides of the equation.
\end{proof}

We get the following variant for $v$ with neat analytic singularities in isolated points.
\begin{lma} \label{lm:gMA-continuity-singular}
	Let $T$ be a closed quasipos $(p,p)$-current with analytic singularities, let $q$ be a quasipsh function with analytic singularities in $Z$, and let $v$ be a quasipsh function with neat analytic singularities in isolated points. %
	We set $q_\kappa:=\max\{q,v-\kappa\}=\max\{q-v,-\kappa\}+v$ %
	such that $q_\kappa$ are quasipsh functions with neat analytic singularities %
	in isolated points
	and $q_\kappa$ is decreasing pointwise to $q$. %
	Then,
	\begin{equation}\label{eq:gMA-continuity-singular}
		[dd^c q]_{dd^c v}\wedge T =
		\lim_{\kappa\rightarrow\infty} dd^c q_\kappa\wedge T.
	\end{equation}
\end{lma}

\begin{proof}%
	As explained in the comment after \autoref{df:qpos-currents-w-analy-sing}, $T$ is locally the difference of two closed positive currents (with analytic singularities). %
	By proving the convergence \eqref{eq:gMA-continuity-singular} for the minuend and subtrahend separately, we may assume that $T$ is positive.
	As $\rho_\kappa=\max\{\cdot,-\kappa\}$ are non-decreasing convex functions decreasing pointwise to the identity as $\kappa\rightarrow\infty$, \autoref{lm:gMA-continuity} implies the weak convergence \eqref{eq:gMA-continuity-singular} on the complement of $L(v)$.
	Fix an $x\in L(v)$.
	Let $\beta$ be an $(n{-}p{-}1,n{-}p{-}1)$-bump form at $x$ such that $\supp dd^c\beta\subset L(v)^c$, see \autoref{rk:BumpFormExists}.
	Fix a constant $C$ such that $\tilde v := \max_\eps \{v,-C\}$ equals $v$ on $\supp dd^c \beta$ ($\max_\eps$ denotes the regularized max for a small enough $\eps$). %
	We set $\tilde q_\kappa:=\max\{q,\tilde v-\kappa\}=\rho_\kappa(q-\tilde v)+\tilde v$. %
	By \autoref{lm:gMA-continuity}, we get that %
	\begin{equation*}
		dd^c \tilde q_\kappa \wedge T\overset{\kappa\rightarrow\infty}\longrightarrow\big(dd^c q \wedge \one_{Z^c}+dd^c \tilde v\wedge \one_Z\big) T.
	\end{equation*}
	Furthermore, $\tilde q_\kappa$ equals $q_\kappa$ on $\supp dd^c \beta$. By applying Stokes and the facts above, we get
	\begin{align}%
		\nonumber
		\int \beta\wedge  dd^c q_\kappa \wedge T
		=& \int dd^c\beta \wedge q_\kappa  T
		= \int dd^c\beta\wedge \tilde q_\kappa T
		=\int \beta\wedge dd^c \tilde q_\kappa \wedge T\\
		\label{eq:convergence-with-bump-forms}
		\overset{\kappa\rightarrow\infty}\longrightarrow&\int \beta\wedge \big(dd^c q \wedge \one_{Z^c}+dd^c \tilde v\wedge \one_Z\big) T %
		=\int dd^c\beta\wedge \big( q \one_{Z^c}+\tilde v\one_Z\big) T\\
		\nonumber
		=&\int dd^c\beta\wedge \big( q\one_{Z^c}+ v \one_Z\big) T
		=\int \beta\wedge \big(dd^c q \wedge \one_{Z^c}+dd^c v\wedge \one_Z\big) T. %
	\end{align}%
	Since locally there is a closed positive $(1,1)$-form $\alpha$ such that $q$, $v$ and $q_\kappa$ are $\alpha$-psh, %
	$S_\kappa:=(dd^cq_\kappa+\alpha)\wedge T$ are positive, and
	we can  apply \autoref{pr:convergence-of-positive-currents}, %
	\ie $S_\kappa$ is converging weakly to a closed positive current $S$ with $S$ equals $[dd^cq]_{dd^cv}\wedge T+\alpha\wedge T$ on the complement of $L(v)$. %
	In particular, we proved the claimed by using \autoref{lm:Unique}.
\end{proof}

Let us conclude this section with the following variant of \autoref{lm:Unique}.

\begin{lma}%
	\label{lm:bump} %
	Let $T$ be a closed quasipos $(p,p)$-current with analytic singularities such that $T$ has support in an analytic set $A$ with  $\codim A\geq p$.
	If for all points $x\in A$, there is an $(n{-}p,n{-}p)$-bump form $\beta$ such that $\int \beta\wedge T =0$, 
	then $T$ equals zero.
\end{lma}
\begin{proof}
	By \autoref{df:currents-w-analy-sing}, %
	we get locally $T=\sum_i T_i$ where
	\begin{equation*}
		T_i=\alpha_i\wedge \one_{W_{i,k_i+1}}dd^c u_{i,k_i}\wedge \one_{W_{i,k_i}}
		\cdots\wedge \one_{W_{i,2}}dd^c u_{i,1} \one_{W_{i,1}}
	\end{equation*}
	with smooth closed real forms $\alpha_i$ of degree $(p-k_i,p-k_i)$, psh $u_{i,j}$ with analytic singularities, and constructible sets $W_{i,j}$ such that $L(u_{i,j})\subset W_{i,j}^c$ for $j\geq 2$.
	If $\codim A> k_i$, then we may assume that $T_i$ vanishes due to $T$'s support (independently of the bump form assumption).
	So, all terms in $T$ vanish except of $T_i$ with $k_i=\codim A=p$, \ie
	\begin{equation*}
		T=\sum\big._i c_i\one_{W_{i,p+1}}dd^c u_{i,p}\wedge \one_{W_{i,p}}
		\cdots\wedge \one_{W_{i,2}}dd^c u_{i,1} \one_{W_{i,1}}.
	\end{equation*}
	As $T$ is quasipos, $c_i\geq 0$, see comment in \autoref{df:qpos-currents-w-analy-sing}. We conclude $T$ is positive. Therefore, \autoref{lm:Unique} implies $T=0$.
\end{proof}

\section{Lelong numbers of closed real currents} %
\label{sc:Lelong}
\noindent
In this section, we are going to study Lelong numbers of closed real currents which are not necessarily positive. To be more precise, we consider currents which are locally the difference of two closed positive ones or currents with analytic singularities, see \autoref{df:currents-w-analy-sing}.
We will recall and generalize basic properties of Lelong numbers following the notation as in the Sections 5ff of \cite{DemaillyAG}*{Chp.~III}; see also \cites{Demailly85, Demailly93}.

\bigskip

Let $X$ be a complex manifold of dimension $n$, and fix a point $x_0\in X$ and coordinates $z$ in a neighbourhood of $x_0$ such shat $z(x_0)=0$.
If $u$ is a (negative) psh function on $X$, the Lelong number of $u$ in $x_0$ is defined by
\begin{equation*}
	\nu(u,x_0) := \liminf_{x\rightarrow x_0}\frac{u(x)}{\log \|z\|}.
\end{equation*}
This definition trivially extends to quasipsh functions $q$. The following lemmata are direct consequences of this definition.
\begin{lma}\label{lm:Lelong-of-restriction}
	If $S\subset X$ is a submanifold \st $q|_S\neq-\infty$, then $\nu(q|_S,x_0)\geq\nu(q,x_0)$. %
\end{lma}
\begin{lma}\label{lm:Lelong-of-pullback}
	If $\pi\colon Y\!\rightarrow X$ is a holom.\ submersion, then $\nu(q\,{\circ}\,\pi,y_0)=\nu(q,\pi(y_0))$. %
\end{lma}
\begin{proof}
	As the statement is local with respect to $y_0$, we may assume that $X\subset\CC^\dimX$, $Y=X\times D$ for $D\subset\CC^m$, $\pi\colon Y\rightarrow X$ is the projection and $y_0=(0,0)$.
	Then, for any sequence $x_k$ in $X$ converging to $0$,
	we get $\frac{q(x_k)}{\log \|x_k\|}=\frac{q(\pi(x_k,0))}{\log \|(x_k,0)\|}$.
	Therefore, $\nu(q,0)\geq\nu(q\circ \pi,(0,0))$.
	If $y_k=(x_k,a_k)$ is a sequence in $Y$ converging to $(0,0)$, then
	$\frac{q(\pi(x_k,a_k))}{\log \|(x_k,a_k)\|} \geq \frac{q(x_k)}{\log \|x_k\|}$ for almost all $k$. %
	This implies the other inequality $\nu(q,0)\leq\nu(q\circ \pi,(0,0))$.
\end{proof}

\medskip

If $T$ is a closed positive $(p,p)$-current, then the Lelong number of $T$ in $x_0$ is given by
\begin{equation}\label{eq:LelongPosCur}
	\nu(T,x_0) = \int \one_{\{x_0\}} \big(dd^c\log \|z\|\big)^{n-p}\wedge T.
\end{equation}
Thereby, %
we use that $L\big(dd^c\log \|z\|\big)=\{x_0\}$ is of codimension $n$ such that the current $\big(dd^c\log \|z\|^2\big)^{n-p}\wedge T$ is defined in the sense of \BTD; see \eg \cite{DemaillyAG}*{Def.~5.4}.
As introduced in \cite{LRSW}*{Sec.~6.4}, we can extend the definition of Lelong numbers to currents which locally are given as differences of two closed positive currents as follows.
\begin{lma}
	Let $T$ be a closed real current on $X$ such that in a neighbourhood of a point $x_0$, $T=T_+-T_-$ for two closed positive currents $T_\pm$. Then, the Lelong number of $T$ in $x_0$ 
	\begin{equation*}
		\nu(T,x_0):=\nu(T_+,x_0)-\nu(T_-,x_0)
	\end{equation*}
	is well-defined since it is independent of the decomposition.
\end{lma}
We call $\nu(T,x_0)$ the Lelong number of $T$ in $x_0$. By this extension, $\nu(\,\cdot\,,x_0)$ is a linear operator on the currents which are given as the difference of two closed positive currents in a neighbourhood of $x_0$.
\begin{proof}
	Let $T$ equal $T_+-T_-$ and $S_+-S_-$ for closed positive currents $T_\pm$ and $S_\pm$ in a neighbourhood of $x_0$. Then, we get 
	$T_++S_- =S_++T_-$ is a closed positive current. As the Lelong number of the sum of closed positive currents is the sum of the Lelong numbers of the summands, we get
	\begin{align*}
		\nu(T_+,x_0)-\nu(T_-,x_0)
		&= \nu(S_++T_+,x_0)-\nu(S_++T_-,x_0) \\
		&= \nu(S_++T_+,x_0)-\nu(S_-+T_+,x_0)=\nu(S_+,x_0)-\nu(S_-,x_0). 
	\end{align*}
	\EoPfwEq
\end{proof}

\smallskip

By Prop.~5.12 in \cite{DemaillyAG}*{Chp.~III},
we have the following semicontinuity property. %

\begin{prop}\label{pr:Lelong-estimate-sequence}
	Let $T$ be a closed current such that $T=T_+-T_-$ for closed positive $T_\pm$ in a neighbourhood of $x_0$.
	$T_\kappa$ be a sequence of closed currents which converges weakly to $T$ such that $T_{\kappa,+}:=T_\kappa+T_-$ are closed positive currents for (almost) all $\kappa$. Then,
		\[\limsup_{\kappa\rightarrow \infty} \nu(T_\kappa,x_0)\leq \nu(T,x_0).\]
\end{prop}

\smallskip

For proper holomorphic submersions, we get the following estimates between the Lelong numbers of closed positive currents and their pushforwards or pullbacks. %

\begin{prop}\label{pr:Lelong-estimate-pushforward}
	Let $\pi\colon X\rightarrow Y$ be a proper holomorphic submersion between complex manifolds $X$ and $Y$, with \dimF-dimensional fibres, and
	let $T$ be a closed positive $(p,p)$-current on $X$
	with $p\geq \dimF$. %
	Then, for all $x_0\in X$, %
	we get %
	\begin{equation*}
		\nu(T,x_0)\leq\nu(\pi_\ast T, \pi(x_0)).
	\end{equation*}
\end{prop}

\begin{proof}
	Let \dimX be the dimension of $X$.
	We may assume that $Y$ is a small enough neighbourhood of $y_0:=\pi(x_0)$ such that $Y\subset \CC^{\dimY}$ with $z'=(z_1\kdots z_{\dimY})$ as coordinates of $Y$ and $z'(y_0)=0$.
	By the assumption on $\pi$, $F:=\pi^{-1}(y_0)$ is a smooth complex manifold of dimension $\dimF$.
	We pick some coordinates  $z''=(z_{\dimY+1}\kdots z_{\dimX})$ of $F$ in a neighbourhood $D$ of $x_0$ such that $z''(x_0)=0$.	
	We may assume that $U=Y\times D$ is a subset of $X$ and $\pi$ is the projection on $Y$ %
	and obtain that $z=(z',z'')$ are coordinates for $U$ with $\pi\circ z = z'$.
	
	We set $\ph(z):=\gamma\log\|z\|$ for constant $\gamma>0$ and $\psi(z'):=\log\|z'\|$ which are both semiexhaustive psh functions defining the Lelong number in $x_0$ and $y_0$, respectively, by using Demailly's generalized Lelong numbers, see \cite{DemaillyAG}*{\S 5 in Chp.~III }, \ie
	\begin{align*}
		\gamma^{\dimX-p}\nu(T,x_0)=\nu(T,\ph)
		&:=\lim_{r\rightarrow-\infty} \nu(T,\ph,r), %
		\ \nu(T,\ph,r)
		 :=\int_{\{\ph<r\}} T\wedge(dd^c \ph)^{\dimX-p}  \hbox{\quad and } 	\\
		\nu(\pi_\ast T,y_0)=\nu(\pi_\ast T,\psi)
		&:=\lim_{r\rightarrow-\infty} \nu(\pi_\ast T,\psi,r), %
		\ \nu(\pi_\ast T,\psi,r)
		:= \int_{\{\psi<r\}} \pi_\ast T\wedge(dd^c \psi)^{\dimX-p}.
	\end{align*}

	We define $\psi_\eps:=\frac12\log\big(\|z'\|^2+\eps\|z''\|^2\big)$ whose sequence is decreasing  pointwise to $\psi$ as $\eps\rightarrow 0$ and $\psi_{\eps,s}:=\max\{\psi_\eps,s\}$ whose sequence is decreasing pointwise to $\psi_s:=\max\{\psi,s\}$ as $\eps\rightarrow 0$ and to $\psi_\eps$ as $s\rightarrow -\infty$.

	Since $\|z'\|^2+\eps\|z''\|^2 \leq \|z\|^2$, we obtain $\gamma\psi_\eps\leq \ph$, \ie $\frac{\ph(z)}{\psi_\eps(z)} \leq \gamma$ close to $x$. Moreover, we have
	\[\gamma\geq \limsup_{\psi_\eps(z)\rightarrow -\infty} \frac{\ph(z)}{\psi_\eps(z)} \geq \limsup_{z'\rightarrow 0} \frac{\ph(z',0)}{\psi_\eps(z',0)} = \gamma.\]
	By Demailly's second comparison theorem for Lelong numbers (\cf Thm~7.8 in \cite{DemaillyAG}*{Chp.~III}), %
	we get that 
	\begin{equation*}
		\nu(T\wedge(dd^c \ph)^{\dimX-p},x_0)\leq \gamma^{\dimX-p}\nu(T\wedge(dd^c \psi_\eps)^{\dimX-p},x_0)<\nu(T\wedge(dd^c \psi_\eps)^{\dimX-p},\ph),
	\end{equation*}
	\ie for all constants $R$ there is an $r<R$ such that 
	\begin{equation}\label{eq:PFpfEstPh}
		\int_{B_r} T\wedge(dd^c \ph)^{\dimX-p} %
		\leq %
		\int_{B_R} T\wedge(dd^c \psi_\eps)^{\dimX-p}
	\end{equation}
	where $B_r=\{\ph<r\}$ and $B_R=\{\ph<R\}$ are  balls with radii $e^r$ and $e^R$, resp. %
	As $\psi_{\eps,s}\searrow \psi_\eps$, we get %
	\begin{equation}\label{eq:PFpfEstPsiEps}
		\int_{B_R} T\wedge(dd^c \psi_\eps)^{\dimX-p}\leq \int_{B_R} T\wedge(dd^c \psi_{\eps,s})^{\dimX-p}
	\end{equation}
	for all $s$ small enough, see for example proof of Thm~3.7 in \cite{DemaillyAG}*{Chp.~III}. %
	Furthermore, 
	\begin{equation}\label{eq:PFpfEstPsiEpsS}
		\begin{split}
		\lim_{\eps\rightarrow 0}\int_{B_R} \hspace{-.5em}T\wedge(dd^c \psi_{\eps,s})^{\dimX-p}
		& = \int_{B_R} \hspace{-0em}T\wedge(dd^c \psi_s\circ\pi)^{\dimX-p}\\
		&\leq \int_{\{\psi\circ\pi < R\}}\hspace{-2.2em} T\wedge(dd^c \psi_s\circ\pi)^{\dimX-p}
		=\int_{\{\psi<R\}}\hspace{-1.5em}\pi_\ast T\wedge(dd^c \psi_s)^{\dimX-p} \\
		&=\int_{\{\psi<R\}}\hspace{-1.5em}\pi_\ast T\wedge(dd^c \psi)^{\dimX-p}
		=\nu(\pi_\ast T,\psi,R)
	\end{split}\end{equation}
	where the second last equation follows from Stokes,
	\cf  Prop.~9.3 in \cite{DemaillyAG}*{Chp.~III}.
	\eqref{eq:PFpfEstPh}, \eqref{eq:PFpfEstPsiEps} and \eqref{eq:PFpfEstPsiEpsS} combined imply
	for every $R$ there is an $r<R$ such that $\nu(T,\ph,r)\leq\nu(\pi_\ast T,\psi,R)$. Since the RHS converges to $\nu(\pi_\ast T,y_0)$ as $R\rightarrow-\infty$ and the LHS converges to $\gamma^{\dimX-p}\nu(T,x_0)$ as $r\rightarrow-\infty$ for every $\gamma<1$, we proved the claimed.
\end{proof}

\begin{prop}[Prop.~5 in \cite{Meo96}]\label{pr:Lelong-current-pullback}
	Let $\pi\colon X\rightarrow Y$ be a proper holomorphic submersion between complex manifolds $X$ and $Y$, and %
	let $T$ be a closed positive $(p,p)$-current on $Y$.
	Then, for all $x_0\in X$, we get
	\begin{equation*}
		\nu(\pi^\ast T,x_0)\geq\nu(T, \pi(x_0)).
	\end{equation*}
\end{prop}
To be more precise, Meo proved that for every holomorphic $f\colon D_1\rightarrow D_2$ with $D_1\subset\CC^n$ and $D_2\subset\CC^m$, and every closed positive $(p,p)$-current $T=\lim_{\kappa\rightarrow\infty}\theta_\kappa$ on $D_2$ with closed positive $(p,p)$-forms $\theta_\kappa$ such that $f^\ast T:=\lim_{\kappa\rightarrow\infty}f^\ast\theta_\kappa$ is a closed positive current on $D_1$, we get $\nu(f^\ast T,x_0)\geq\nu(T, f(x_0))$.

\bigskip

\subsection*{Lelong numbers of currents with analytic singularities}
\begin{rem-def}\label{df:Lelong-for-analytic-sing}
	If $T$ is a closed positive $(p,p)$-current and $\alpha$ a closed strongly positive $(k,k)$-form ($k>0$), then $\alpha\wedge T$ is a closed positive current by definition.
	For every point $x\in X$, we get that
	\begin{align*}
		\nu(\alpha\wedge T,x)
		= & \int \one_{\{x\}} \big(dd^c\log \|z\|\big)^{n-p-k}\wedge \alpha \wedge T\\
		= & \int \alpha \wedge \one_{\{x\}} \big(dd^c\log \|z\|\big)^{n-p-k}\wedge  T = 0
	\end{align*}
	where $z$ denotes coordinates of $X$ with $z(x)=0$.
	The last equation follows from that a closed positive $(s,s)$-current with support in an analytic set of codimension strictly greater than $s$ vanishes.

	Therefore, if $\alpha$ is real (not necessary strongly positive), we may trivially extend the definition of Lelong numbers to currents as $\alpha\wedge T$, and their linear combinations. In particular, this defines the Lelong number for every closed current $T$ with analytic singularities.
	We get
	\begin{align*}
		\nu(T,x)
		=&\ \nu\Big(\sum_i c_i\cdot \one_{W_{i,p+1}}dd^c u_{i,p}\wedge \one_{W_{i,p}} \cdots\wedge \one_{W_{i,2}}dd^c u_{i,1},x\Big)\\
		=& \sum_i c_i\cdot \nu\big( \one_{W_{i,p+1}}dd^c u_{i,p}\wedge \one_{W_{i,p}} \cdots\wedge \one_{W_{i,2}}dd^c u_{i,1},x\big)
	\end{align*}
	for constants $c_i$, psh functions $u_{i,1}\kdots u_{i,p}$ with analytic singularities and (suitable) constructible sets $W_{i,2}\kdots W_{i,p+1}$ (defined in a neighbourhood of $x$).
	This is clear since the Lelong numbers of other terms in the local form of $T$ given by \eqref{eq:local-form-of-currents-with-as} vanish.
\end{rem-def}

\medskip

The following is a special variant of Cor.~7.9 in \cite{DemaillyAG}*{Chp.~III} inferred from Demailly's second comparison theorem. %
\begin{prop}\label{pr:Lelong-estimate-AW-products}
	Let $X$ be a complex manifold of dimension \dimX,
	let $T$ be a closed quasipos $(p,p)$-current with analytic singularities for $p<\dimX$, and
	let $q$ be a quasipsh function with analytic singularities in $Z$. For every point $x\in X$, we get
	\begin{equation}\label{eq:LelongEstAW}
		\nu(dd^c q\wedge \one_{Z^c}T,x)\geq \nu(q,x)\cdot\nu(\one_{Z^c}T,x).
	\end{equation}
	If $\eta$ is a closed quasipos (1,1)-current with analytic singularities in isolated points, then
	\begin{equation}\label{eq:LelongEstgMA}
		\nu([dd^c q]_{\eta}\wedge T,x)\geq\min\{\nu(q,x),\nu(\eta,x)\}\cdot\nu(T,x).
	\end{equation}
\end{prop}

\begin{rem}\label{rk:Lelong-estimate-AW-products-needs-eta}
	The cutoff $\one_{Z^c}$ in the RHS of \eqref{eq:LelongEstAW} cannot be omitted as the following example shows. Let $f$ be holomorphic, let $\sigma$ be smooth, and set $q:=\log |f|^2+\sigma$. Then
	\begin{equation*}
		\nu((dd^c q)^2_\AW,0)=\nu(dd^c q\wedge\one_{Z^c}dd^c q,0)=\nu(\big([f=0]+dd^c \sigma\big)\wedge dd^c \sigma,0)=0
	\end{equation*}
	while $\nu(dd^c q,0)=\nu([f=0],0) \neq 0$.
\end{rem}

\begin{proof}[Proof of \autoref{pr:Lelong-estimate-AW-products}.]
	The proof follows Demailly's strategy adapted to our setting.

	Without loss of generality, we may assume that $T$ is positive since the Lelong number of
	each non-positive term in the local form of $T$ given by \eqref{eq:local-form-of-currents-with-as} vanishes, see \autoref{df:Lelong-for-analytic-sing}.

	Let us assume that we work on a coordinate patch of $X$ with $x=0$ and $z$ denoting the coordinates, and that $q=u+\sigma$ for a psh $u$ with analytic singularities and smooth function $\sigma$. %
	Due to the homogeneity of the Lelong number, we may assume that $\nu(u,0)=\nu(q,0)=1-\eps$ by rescaling $q$, and 
	to prove \eqref{eq:LelongEstAW}, it is enough to show
	\begin{equation}\label{eq:redLelongEstAW}
		\nu(\one_{Z^c}T,0)\leq\nu(dd^c u\wedge \one_{Z^c}T,0)=\nu(dd^c q\wedge \one_{Z^c}T,0).
	\end{equation}
	We set $v:=\log \|z\|$ and $u_\kappa:=\max\{u,v-\kappa\}=\max\{u-v,-\kappa\}+v$.
	Since $\nu(u,0)<\nu(v,0)$, there is an $r_0>0$ such that $v-\kappa\geq u$, \ie $u_\kappa=v-\kappa$ on $\{\|z\|<r_0\}$. %
	In particular, $\nu(u_\kappa,0)=$ $\nu(v,0)=1$ and
	\begin{equation*}
		\nu(dd^c u_\kappa\wedge T,0)=\nu(dd^c v\wedge T,0),
	\end{equation*}
	whereby the wedge product is well-defined due to \BTD.
	On the other side, $u_\kappa \searrow u$ pointwise as $\kappa\rightarrow\infty$ such that by \autoref{lm:gMA-continuity-singular}, we get the weak convergence
	\begin{equation*}
		dd^c u_\kappa\wedge T \overset{\kappa\longrightarrow\infty}{\longrightarrow} dd^c u\wedge \one_{Z^c} T + dd^c v\wedge \one_{Z} T.
	\end{equation*}
	As the Lelong number is upper semicontinuous in its first argument (see \autoref{pr:Lelong-estimate-sequence}), %
	we get
	\begin{equation*}
		\nu(dd^cv\wedge T,0)=\nu(dd^cu_\kappa\wedge T,0)\leq \nu(dd^cu\wedge\one_{Z^c} T,0)+\nu(dd^cv\wedge\one_{Z} T,0).
	\end{equation*}
	By subtracting $\nu(dd^cv\wedge\one_{Z} T,0)$ on both sides, we get \eqref{eq:redLelongEstAW} since $\nu(\one_{Z^c} T,0)=\nu(dd^cv\wedge \one_{Z^c}T,0)$, see \eqref{eq:LelongPosCur}. %
	\medskip

	As $[dd^c q]_{\eta}\wedge T=dd^c q\wedge\one_{Z^c} T + \eta\wedge \one_Z T$ %
	and the Lelong number is linear (in its first argument), \eqref{eq:LelongEstgMA} follows from \eqref{eq:LelongEstAW} and 
	\begin{equation*}
		\nu(\eta\wedge \one_Z T,0)\geq \nu(\eta,0)\cdot\nu(\one_Z T,0).
	\end{equation*}
	which also follows from \eqref{eq:LelongEstAW} since 	%
	$\nu(\one_{X\minus 0}\one_Z T,0)=\nu(\one_Z T,0)$ as $p<n$.
\end{proof}

Combining \autoref{pr:Lelong-estimate-pushforward} with \autoref{pr:Lelong-estimate-AW-products}, we get the following corollary.

\begin{cor}\label{cr:Lelong-estimate-pushforwards-AW-products}
	Let $X$ be a complex manifold of dimension \dimX, and let $\pi\colon X\rightarrow Y$ be a proper holomorphic submersion with \dimF-dimensional fibres.
	Let $T$, $\theta$ and $\eta$ be closed positive currents with analytic singularities on $X$,  
	such that $T$ is of bidegree $(p,p)$, $\theta$ and $\eta$ are of bidegree $(1,1)$ and in the same \BC-class, $L(\eta)$  is a set of isolated points.
	Then, for all $k$ with $\dimF\leq k+p\leq n$, and for all $x\in X$,
	\begin{equation}\label{eq:corLEPFAW}
		\nu\big(\pi_\ast\big([\theta]_{\eta}^k\wedge T\big),\pi(x)\big)\geq\min\{\nu(\theta,x),\nu(\eta,x)\}^k\cdot \nu(T,x).
	\end{equation}
\end{cor}

\begin{rem}\label{rk:pushforward-Lelong-estimate-fails-for-non-pos}
	In general, \eqref{eq:corLEPFAW} cannot hold if we consider quasipos $\theta$ which are not positive near $\pi^{-1}(\pi(x))$. This is obvious since the RHS in \eqref{eq:corLEPFAW} is always positive for such $\theta$ while the LHS can be negative as in the following simple example.

	Let $Y$ be \CC with coordinate $z$, let $F$ be $\PP^1$ with the coordinates $[\xi_1\,{:}\,\xi_2]$, let $X$ be $\CC\times\PP^1$, let $\pi\colon X\rightarrow Y$ be the projection on the first component, and
	let $\omega_\FS$ be the pullback to $X$ of the Fubini-Study form on $\PP^1$.
	Then, $\theta=[z=0]-\omega_\FS$ is quasipos and $\pi_\ast([\theta]^2_\eta)=\pi_\ast([\theta]^2)
	=\pi_\ast(-\omega_\FS)\wedge[z=0]=-[z=0]$ (for any $\eta$) such that $\nu(\pi_\ast([\theta]^2),0)=-1$. %
\end{rem}

\begin{rem}\label{rk:pushforward-Lelong-not-indep-of-eta}
	If $\alpha_1,\alpha_2\in\{\theta\}_{\BC}$ are smooth, then 
	$\nu\big(\pi_\ast([\theta]_{\alpha_1}^k\wedge T),y\big)=\nu\big(\pi_\ast([\theta]_{\alpha_2}^k\wedge T),y\big)$ following the argumentation in the proof of \cite{LRSW}*{Thm~1.1 (3)}. As the following counter example shows, this is not correct if one of them has neat analytic singularities in isolated points.

	Let $Y$ be $\CC$ with the coordinate $z$, let $X$ be $\CC\times\PP^1$ with the coordinate $[\xi_1\,{:}\,\xi_2]$, and let $\pi\colon X\rightarrow Y$ be the projection on the first component.
	Consider the current $\theta=[\xi_2=0]$ which has analytic singularities in $Z=\{\xi_2=0\}$, $\eta:=\frac12dd^c\log\big(|z|^2|\xi_1|^2+|\xi_2|^2\big)$ which has only one neat analytic singularity in $(0,[1\,{:}\,0])$, and $\alpha=\omega_\FS=\frac12dd^c\log\big(|\xi_1|^2+|\xi_2|^2\big)$ the Fubini-Study form on $\PP^1$. Then, $\eta-\alpha=dd^c\log\frac{|z|^2|\xi_1|^2+|\xi_2|^2}{|\xi_1|^2+|\xi_2|^2}$.
	For any ball $B\subset\CC$, we get that
	\begin{equation*}
		\int_B\pi_\ast\big([\theta]_{\eta}^2\big)-\int_B\pi_\ast\big([\theta]_{\alpha}^2\big)=\int_B\pi_\ast\big((\eta-\alpha)\wedge\one_Z[\xi_2=0]\big)
		=\int_{B\times\{[1:0]\}}\hspace*{-1em}(\eta-\alpha) %
		=\int_{B} dd^c\log{|z|}.
	\end{equation*}
	Therefore, $\nu\big(\pi_\ast\big([\theta]_{\eta}^2\big),0\big)-\nu\big(\pi_\ast\big([\theta]_{\alpha}^2\big),0\big)=\nu(\log|z|,0)=1$.
\end{rem}

At least, we obtain that the Lelong numbers of currents $\pi_\ast([\theta]_{\eta}\wedge T)$ are independent of $\eta$ among all quasipos $(1,1)$-currents with the same type of singularities as follows. 
\begin{prop}\label{pr:pushforward-Lelong-indep-of-sp-eta}
	Let $\pi\colon X\rightarrow Y$ be a proper holomorphic submersion,
	let $T$ and  $\theta$ be closed quasipos current with analytic singularities of bidegree $(p,p)$ and $(1,1)$, resp.  %
	Let $\eta_1, \eta_2$ be two quasipos $(1,1)$-currents  with neat analytic singularities in isolated points both in the same \BC-class as $\theta$ such that $\nu(\eta_1,x)=\nu(\eta_2,x)$ for all $x\in X$. %
	Then, for all $y\in Y$, we get
	\begin{equation*}
		\nu\big(\pi_\ast([\theta]_{\eta_1}^k\wedge T),y\big)=\nu\big(\pi_\ast([\theta]_{\eta_2}^k\wedge T),y\big).
	\end{equation*}
\end{prop}
\begin{proof}
	As $\eta_1$ and $\eta_2$ are in the same \BC-class, we get that $\eta_1-\eta_2=dd^c q$ for a function $q$ which is locally the difference of two quasipsh $q_1$ and $q_2$ function with neat analytic singularities in isolated points (given by $\eta_1=dd^c q_1$ and $\eta_2=dd^c q_2$).
	As $\nu(q_1,x)=\nu(q_2,x)$, we obtain that $q_1-q_2=q$ is smooth.
	Now, we can repeat the argumentation in the proof of \cite{LRSW}*{Thm~1.1 (3)} to prove the claimed. %
\end{proof}

\section{Proofs of the main results}
\label{sc:proofs}
\noindent
Before proving the second main result of the present work, let us observe the following.

\let\dimYBU\dimY
\def\dimY{\ensuremath{\tilde n}}
\begin{lma}\label{lm:peaker}
	Let $(X,\omega)$ be a Kähler manifold. For each relatively compact $U\subset X$, there exists a constant $\delta=\delta(\omega)$ such that for every $x\in U$ and positive $\tilde\delta\leq \delta$, there is a closed positive $(1,1)$-current $\eta_x\in\{\omega\}$ with neat analytic singularities only in $x$ and $\nu(\eta_x,x)=\tilde\delta$. %

	If $X\subset B\times \PP^m$ for $B\subset \CC^{\dimY}$  %
	and if $\omega$ is given as restriction of the product of the Euclidean and the Fubini-Study metric,
	then we can take $\delta= 1$. %
\end{lma}

\begin{proof}
	Set $\dimX:=\dim X$. For every $x\in U$, 
	let $B\subset U$ be a neighbourhood  of $x$ with coordinates $z=(z_1\kdots z_\dimX)$ such that $B\subset\CC^\dimX$ with $z(x)=0$, and let $\chi$ be a cutoff function on $B$ which is identically $1$ near $x$.
	The function $q_x(z):=\chi(z)\log\|z\|$ is a quasipsh function on $B$ which extends trivially to $U$. There is a constant $\delta_x>0$ such that $\delta_x dd^cq_x+\omega\geq 0$.
	We can choose $\delta_x$ continuously depending of $x$.
	As $U$ is relatively compact in $X$, we get $\delta:=\min_{x\in U} \delta_x>0$.
	Then, $\eta_x:=\tilde\delta dd^cq_x+\omega$ has the claimed properties for all $\tilde\delta\leq\delta$.	

	For $X=B\times\PP^\dimF$ with the coordinates $z'$ for $B$ and $[\xi_1\,{:}\ldots{:}\,\xi_{m+1}]$ for $\PP^\dimF$, we may assume that $x=(0,[1\,{:}\,0\,{:}\ldots{:}\,0])$.
	Then, $u_x(z',[\xi])=\frac12\log\big(\|z'\|^2|\xi_1|^2+|\xi_2|^2+\ldots+|\xi_{m+1}|^2\big)+\|z'\|^2$ is psh. So, $dd^cu_x$ is a well-defined positive current on $X$, in the same cohomology class as $\omega$, and $\nu(dd^cu_x,x)=1$.
	In particular, $\eta_x=\tilde\delta dd^cu_x+(1-\tilde\delta)\omega$ has the claimed properties.

	For $X\subset B\times \PP^\dimF$, we set 
	$\tilde u_x=(\pr_2^\ast{u_x})|_X$ where $\pr_2\colon B\times\PP^\dimF\rightarrow \PP^\dimF$ denotes the projection on $\PP^\dimF$.
	Then, $\eta_x=\tilde\delta\mu{\cdot}dd^c \tilde u_x+(1-\tilde\delta\mu)\omega$ has the claimed properties for $\mu=\nu(u_x,x)/\nu(\tilde u_x,x) \leq 1$ (for the last inequality, see \autoref{lm:Lelong-of-restriction} and \autoref{lm:Lelong-of-pullback}).
\end{proof}
\let\dimY\dimYBU

\RepeatSecondMainTheorem

\begin{proof}%
	By the assumptions, there are Kähler forms $\omega_i$ on $U=\pi^{-1}(V)$ such that $\theta_i$ and $\omega_i$ are in the same cohomology class on $U$.
	Therefore, \autoref{lm:peaker} gives us positive constants $\delta_i$ and, for every $x\in\pi^{-1}(y)$, closed positive $(1,1)$-currents $\tilde\eta_i\in\{\omega_i\}_{\BC}=\{\theta_i\}_{\BC}$ with neat analytic singularities only in $x$ on $U$ such that $\nu(\tilde\eta_i,x)=\delta_i$.
	In particular, $\delta_i$ and $\tilde\eta_i$ only depend on the cohomology class represented by $\omega_i$.

	Let us extend $\tilde\eta_i$ to the whole $X$.
	Pick a closed real $(1,1)$-form  $\alpha_i$ on $X$ in the same \BC-class as $\theta_i$.
	In particular, %
	there is a quasipsh function $q_i$ with neat analytic singularities in isolated points on $U$
	such that $dd^c q_i=\tilde\eta_i-\alpha_i$.
	Let $\chi$ be a cutoff function with support on $V$ and identically $1$ in a neighbourhood of $y$. We define $\eta_i:=dd^c(\chi\circ\pi\cdot q_i)+\alpha_i$ which is in the same \BC-class as $\theta_i$ and $\eta_i=\tilde\eta_i$ on a neighbourhood of $\pi^{-1}(y)$.
	Furthermore, $\eta_i$ is positive there and $\nu(\eta_i,x)= \delta_i$.%
	
	Following \autoref{df:qpos-currents-w-analy-sing}, we get that 
	$T:=[\theta_1]_{\eta_t}^{k_t}\wedge \cdots \wedge[\theta_1]_{\eta_1}^{k_1}$ is a closed quasipos current with analytic singularities which is positive in a neighbourhood of $\pi^{-1}(y)$.
	In particular, \autoref{pr:Lelong-estimate-pushforward} implies $\nu(\pi_\ast T,y)\geq \nu(T,x)$ as long $\dimF\leq k_1+\ldots+k_t\leq \dimX$. %
	By \autoref{pr:Lelong-estimate-AW-products}, we get that
	$\nu(T,x)\geq \prod_i\min\{\nu(\theta_i,x),\delta\}^{k_i}$. This proves (i).
	\medskip

	(ii) is now a direct consequence of \autoref{pr:pushforward-Lelong-indep-of-sp-eta} (recursively applied): We can exchange the quasipositive $\eta_i$ in the current on the LHS of \eqref{eq:mainthm-2bII} by a positive $\tilde\eta_i$ with the same Lelong number in $x$ which is given by \autoref{lm:peaker}.
\end{proof}
\bigskip

Let us continue with the proof of the first main result of the present work.

\RepeatFirstMainTheorem

\begin{proof}%
	Since the statement is local with respect to neighbourhoods of $y\in Y$,
	and since $\theta_i$ is in a Kähler class near $\pi^{-1}(y)$ by assumption,
	we may assume that $X$ is Kähler by shrinking $Y$ and $X=\pi^{-1}(Y)$.
	Let $\omega$ denote a Kähler form on $X$. %
	By \autoref{lm:peaker}, there are constants $\delta_i>0$ and for each $x\in\pi^{-1}(y)$, closed positive $(1,1)$-currents $\eta_i=\eta_{i,x}\in\{\tilde\theta_{i}\}_{\BC}$ %
	with neat analytic singularities only in $x$.

	\smallskip

	\ManualItem{i} By shrinking %
	$Y$ and $X=\pi^{-1}(Y)$ again, Demailly's regularization result (see \autoref{thm:DemRegAnalySing}) implies that
	there are positive constants $\eps_\kappa$ %
	and closed $\eps_\kappa\omega$-positive $(1,1)$-currents $\theta_{i,\kappa}\in\{\theta_i\}_{\BC}$ with analytic singularities such that $\theta_{i,\kappa}$ is decreasing pointwise to $\theta_i$, $\nu(\theta_{i,\kappa},x)\rightarrow \nu(\theta_{i},x)$ and $\eps_\kappa\rightarrow 0$ as $\kappa\rightarrow\infty$. %
	We set $\tilde\theta_{i,\kappa}:=\theta_{i,\kappa}+\eps_\kappa\omega$ %
	and $\tilde\eta_{i,\kappa}:=\eta_i+\eps_\kappa\omega\in\{\tilde\theta_{i,\kappa}\}_{\BC}$.
	As $\tilde\theta_{i,\kappa}$ are positive, we can apply \autoref{MainThm:analytic-singularities} such that
	\begin{equation}
		\label{eq:first-inequ-to-get-first-main-result}
		\begin{split}
		\nu\Big(\pi_\ast \big( [\tilde\theta_{t,\kappa}]_{\tilde\eta_{t,\kappa}}^{k_t}\!\!\wedge \cdots \wedge[\tilde\theta_{1,\kappa}]_{\tilde\eta_{1,\kappa}}^{k_1}\big), y\Big)
		&\geq\prod\big._{i=1}^t \min\{\nu(\tilde\theta_{i,\kappa},x),\nu(\tilde\eta_{i,\kappa},x)\}^{k_i}\\ %
		& = \prod\big._{i=1}^t \min\{\nu(\theta_{i,\kappa},x),\delta_i\}^{k_i}.%
	\end{split}\end{equation}
	By the Theorems \ref{thm:extBTD-continuity} and \ref{thm:extBTD-coincides-genMA}, we obtain that
	\begin{equation*}
		\pi_\ast \big( [\theta_{t}]^{k_t}\wedge \cdots \wedge[\theta_{1}]^{k_1}\big)=
		\lim_{\kappa\rightarrow\infty}\pi_\ast \big( [\theta_{t,\kappa}]_{\eta_t}^{k_t}\wedge \cdots \wedge[\theta_{1,\kappa}]_{\eta_1}^{k_1}\big).
	\end{equation*}
	Since $[\tilde\theta_{i,\kappa}]_{\tilde\eta_{i,\kappa}}\wedge\cdot=\big([\theta_{i,\kappa}]_{\eta_{i}}+\eps_\kappa \omega\big)\wedge\cdot $, we get
	\begin{equation*}
		\begin{split}
			[\tilde\theta_{i,\kappa}]_{\tilde\eta_{i,\kappa}}^{k_i}\wedge T
			=&\sum\big._{l=0}^{k_i} (-1)^l \smbinom{k_i}{l}[\theta_{i,\kappa}]_{\eta_{i}}^{k_i-l}\wedge\eps_\kappa^l\omega^{l}\wedge T\\
			=&[\theta_{i,\kappa}]_{\eta_{i}}^{k_i}\wedge T+
			\sum\big._{l=1}^{k_i} (-1)^l \smbinom{k_i}{l}[\theta_{i,\kappa}]_{\eta_{i}}^{k_i-l}\wedge\eps_\kappa^{l}\omega^{l}\wedge T
		\end{split}
		\end{equation*}
	for any closed current $T$ with analytic singularities.
	By applying this recursively to each factor in $[\tilde\theta_{t,\kappa}]_{\tilde\eta_{t,\kappa}}^{k_{t-1}}\wedge \cdots \wedge[\tilde\theta_{1,\kappa}]_{\tilde\eta_{1,\kappa}}^{k_1}$, we  get 
	\begin{equation*}
		\pi_\ast \big( [\tilde\theta_{t,\kappa}]_{\tilde\eta_{t,\kappa}}^{k_t}\wedge \cdots \wedge[\tilde\theta_{1,\kappa}]_{\tilde\eta_{1,\kappa}}^{k_1}\big)=\pi_\ast \big( [\theta_{t,\kappa}]_{\eta_t}^{k_t}\wedge \cdots \wedge[\theta_{1,\kappa}]_{\eta_1}^{k_1}\big)+\sum\big._{j}\eps_\kappa^{l_j} T_{j,\kappa}
	\end{equation*}
	for positive integers $l_j>0$ and currents $T_{j,\kappa}$ which are pushforwards of quasipos currents with analytic singularities.
	Moreover, we can apply \autoref{thm:extBTD-continuity} for each term and get
	\begin{equation*}
		\pi_\ast \big( [\theta_{t}]^{k_t}\wedge \cdots \wedge[\theta_{1}]^{k_1}\big)=
		\lim_{\kappa\rightarrow\infty}\pi_\ast \big( [\tilde\theta_{t,\kappa}]_{\tilde\eta_{t,\kappa}}^{k_t}\wedge \cdots \wedge[\tilde\theta_{1,\kappa}]_{\tilde\eta_{1,\kappa}}^{k_1}\big).
	\end{equation*}
	As all currents are positive, we can estimate the Lelong number of the LHS from below with the limit superior of the Lelong number of the RHS, see \autoref{pr:Lelong-estimate-sequence}.
	This and \eqref{eq:first-inequ-to-get-first-main-result} imply
	\begin{equation*}
			\nu\Big(\pi_\ast \big( [\theta_t]^{k_t}\wedge \cdots \wedge[\theta_1]^{k_1}\big), y\Big)
			\geq
			\lim_{\kappa\rightarrow\infty}\prod\Big._{i=1}^t \min\{\nu(\theta_{i,\kappa},x),\delta_i\}^{k_i}
			=
			\prod\Big._{i=1}^t \min\{\nu(\theta_i,x),\delta_i\}^{k_i}
	\end{equation*}
	where the last equation follows from the choice of $\theta_{i,\kappa}$, \cf \autoref{thm:DemRegAnalySing} (ii).

	\smallskip

	\ManualItem{ii}
	We prove \eqref{eq:mainthm-1bII} by induction over $t$.
	The case $t=1$ follows from (i). For the induction step, we will prove
	\begin{equation}\label{eq:IndVers-MTii}
		\nu\big(\pi_\ast\big([\theta_t]^{k_t}\big)\wedge T,y\big)\geq \min\{\nu(\theta_t,x),\delta_t\}^{k_t}\cdot\nu(T,y)
	\end{equation}
	for the closed positive current $T:=\pi_\ast \big([\theta_{t-1}]^{k_{t-1}}\big)\wedge \cdots \wedge \pi_\ast \big([\theta_{1}]^{k_1}\big)$.

	Pick a closed real $(1,1)$-form $\alpha\in\{\theta_{t}\}_{\BC}=\{\eta_{t}\}_{\BC}$. 
	There are $\alpha$-psh functions $q$ and $v$ (the latter with neat analytic singularities only in $x$) such that $\theta_{t}=dd^c q+\alpha$ and $\eta_{t}=dd^c v+\alpha$. We set $q^{(\lambda)}:=\max\{q,\frac\delta{\delta_t} v-\lambda\}$ with $\delta<\nu(q,x)=\nu(\theta_t,x)$ and $\delta\leq\delta_t$. We get  $q^{(\lambda)}$ is $\alpha$-psh and decreasing pointwise to $q$ as $\lambda\rightarrow\infty$.
	As in the proof of \autoref{pr:Lelong-estimate-AW-products}, %
	$q^{(\lambda)}$ has neat analytic singularities only in $x$ and $\nu(q^{(\lambda)},x)=\delta$. %
	We set $\theta_{t}^{(\lambda)}=dd^c q^{(\lambda)}$ and get %
	\begin{equation*}%
		\pi_\ast\big([\theta_{t}]^{k_t}\big)\wedge T
		=\lim_{\lambda\rightarrow\infty}
		\pi_\ast\Big(\big[\theta_{t}^{(\lambda)}\big]^{k_t}\Big)\wedge T
		=\lim_{\lambda\rightarrow\infty}
		\pi_\ast\Big(\big(\theta_{t}^{(\lambda)}\big)^{k_t}\wedge \pi^\ast T\Big)
	\end{equation*}
	where the first equation follows from \autoref{thm:extBTD-continuity} %
	and the second from \autoref{pr:extBTD-def} and \BTD (by selecting a smooth approximation of $\theta_{t}^{(\lambda)}$ for each $\lambda$). %
	Semicontinuity of the Lelong number, see \autoref{pr:Lelong-estimate-sequence}, implies
	\begin{equation*}
		\nu\big(\pi_\ast\big([\theta_t]^{k_t}\big)\wedge T,y\big)
		\geq 
		\limsup_{\lambda\rightarrow\infty}
		\nu\Big(\pi_\ast\Big(\big(\theta_{t}^{(\lambda)}\big)^{k_t}\wedge \pi^\ast T\Big),y\Big).
	\end{equation*}
	By the Propositions \ref{pr:Lelong-estimate-pushforward} and \ref{pr:Lelong-current-pullback}, we get 
	\begin{equation*}
		\nu\Big(\pi_\ast\Big(\big(\theta_{t}^{(\lambda)}\big)^{k_t}\wedge \pi^\ast T\Big),y\Big)
		\geq \nu\big((\theta_{t}^{(\lambda)}\big),x)^{k_t}\cdot\nu(\pi^\ast T,x)
		\geq \delta^{k_t}\cdot\nu(T,y).
	\end{equation*}
	Since this holds for all $\delta\leq\delta_t$ with $\delta<\nu(\theta_t,x)$, \eqref{eq:IndVers-MTii} is proven.
\end{proof}

\section{Segre currents}
\label{sc:Segre}
\noindent
In \cite{LRRS}, Lärkäng, Raufi, Ruppenthal and the author define so-called Chern and Segre currents for singular Hermitian metrics on vector bundles  which are positively curved in the sense of Griffiths\footnote{as defined by \cite{BP08}*{Def. 3.2}; also called singularly (semi-) positive in the sense of Griffiths}\hspace{-.9ex}. These currents represent the Chern and Segre classes of the vector bundles. In this section, we are going to present the construction of these currents in a (slightly) more general setting.

\begin{set}\label{st:section-4}
	Let $X$ be a complex manifold of dimension \dimX, let $E$ be a holomorphic vector bundle on $X$ of rank \rkE,
	let $\pi\colon \PE\rightarrow X$ denote the projectivization of hyperplanes of $E$ (\ie $\pi^{-1}(x)=\PP(E^\ast_x)$ for all $x\in X$),
	let $L:=\cO_{\PP(E)}(1)\rightarrow\PE$ denote the hyperplane bundle of $E$, and %
	let $e^{-\ph}$ be a singular metric on $L$ such that its curvature current $\theta:=dd^c\ph$  is quasipositive, \ie the weights $\ph$ are quasipsh functions.
	Let $L(\ph)$ denote the unbounded locus of $\ph$, \ie $L(\ph)=L(\theta)$.
	There is a (pos) $(1,1)$-form $\gamma$ on \PE such that $\theta+\gamma\geq 0$, \ie $\theta$ is $\gamma$-pos.
	For any small enough open set $U\subset X$, we may pick $\gamma$ such that $\gamma$ is closed and positive on $\pi^{-1}(U)$, see \autoref{rk:Kaehler-pi}.
	Let $e^{-\ph_\kappa}$ be a sequence of smooth/singular metrics on $L$ such that $\ph_\kappa$ is decreasing pointwise to $\ph$ and $dd^c\ph_\kappa$ is $\gamma$-positive.
	Let $T$ be a closed real $(p,p)$-current on $X$ which is locally given as difference of two closed positive currents.
\end{set}

If $H$ is a singular Hermitian metric on $E$, then it induces a singular metric $e^{-\psi}$ on $L$ whose dual metric is given by $e^\psi=\pi^\ast H^\ast|_{L^\ast}$ since $L^\ast=\cO_{\PE}(-1)$ is a subbundle of $\pi^\ast E^\ast$. %
$e^{-\psi}$ has a quite special geometric behaviour on each fibre $\pi^{-1}(x)$ coming from the Hermitian structure of $H_x$ which turns out to be unnecessary for the definition of Chern and Segre currents.
To keep the setting more general, we consider any singular metric $e^{-\ph}$ on $L$. %
There is a one-to-one correspondence between singular metrics on $L$ and singular metrics on $E^\ast$ considering so-called singular Finsler metrics\footnote{A singular Finsler metric $h$ on $E^\ast$ is given by $E^\ast\rightarrow [0,\infty], (p,\xi)\mapsto \Vert\xi\Vert_{h(p)}$ with $\Vert\lambda \xi\Vert_{h(p)}=|\lambda|{\cdot}\Vert\xi\Vert_{h(p)}$.}\hspace{-.9ex}.

If the weights of $\ph$ are psh, \ie $dd^c\ph$ is positive, then $L=\cO_{\PP(E)}(1)$ is pseudoeffective, and so is $E$ by definition%
\footnote{In the literature, sometimes called weakly to distinguish it from strongly pseudoeffective.}\hspace{-.9ex}.
Moreover, we get that $E$ is strongly pseudoeffective if $\pi(L(\ph))$ does not equal $X$ by definition, see \cite{BDPP13}*{Def.~7.1}. In particular, the following results give Segre and Chern currents for strongly pseudoeffective vector bundles and not only vector bundles positively curved in the sense of Griffiths. Let us stress that we work in the even more general \autoref{st:section-4} where the weights $\ph$  are just supposed to be quasipsh.

\bigskip

As in \cite{LRRS}*{Prop.~4.6}, \autoref{pr:extBTD-def} implies the following result.
\begin{thm}\label{thm:def-SegreCurrents-general}
	In the \autoref{st:section-4} with $\ph_\kappa$ smooth,
	if $\pi(L(\ph))$ is contained in an analytic set of codimension $\geq k+p$, then 
	\begin{equation}\label{eq:def-SegreCurrents-general}
		s_k(E,\ph)\wedge T: = (-1)^k \pi_\ast \big([dd^c\ph]^{k+r-1}\wedge \pi^\ast T\big)\overset{\Def}{=} (-1)^k \lim_{\kappa\rightarrow\infty} \pi_\ast \big((dd^c\ph_\kappa)^{k+r-1}\wedge \pi^\ast T\big)
	\end{equation}
	is a well-defined closed real current which is locally the difference of two closed positive currents, and independent of the choice of the smooth approximation $\ph_\kappa$.
\end{thm}

For a smooth Hermitian metric $H$ on $E$, 
the total Segre form $s(E,H)$ is (classically) defined as
the inverse of the total Chern form $c(E,H)=1+c_1(E,H)+\ldots+c_n(E,H)$.
Then, its $(k,k)$-component $s_k(E,H)$ can be calculated using the projectivization of the vector bundle, $s_k(E,H)=(-1)^k\pi_\ast(dd^c \psi)^{k+r-1}$ where $e^\psi=\pi^\ast H^\ast|_{L^\ast}$.
Therefore, the Segre currents $s_k(E,\ph)$ defined above are naturally extending the concept of Segre forms to singular metrics on $\cO_\PE(1)$. Moreover, the so-called Chern currents 
\begin{equation}\label{eq:def-ChernCurrents-general}
	c_k(E,\ph) := (-1)^t\sum_{k_1+\ldots+k_t=k} s_{k_t}(E,\ph)\wedge \cdots \wedge s_{k_1}(E,\ph)
\end{equation}
given by a recursive application of \autoref{thm:def-SegreCurrents-general}
are naturally extending the concept of Chern forms.

\medskip 

\autoref{cr:extBTD-continuity-special} implies the following monotone continuity result generalizing \cite{LRRS}*{Thm~1.5}.
\begin{thm}\label{thm:wedge-products-of-segre-currents}
	In the \autoref{st:section-4},
	if $\pi(L(\ph))$ is contained in an analytic set of codimension $\geq k_1+\ldots+k_t$, then
	\begin{equation*}
		s_{k_t}(E,\ph)\wedge \cdots \wedge s_{k_1}(E,\ph) = \lim_{\kappa\rightarrow\infty} \pi_\ast \big((dd^c\ph_\kappa)^{k_t+r-1}\big)\wedge \cdots \wedge \pi_\ast \big((dd^c\ph_\kappa)^{k_1+r-1}\big)
	\end{equation*}
	for any (not necessarily smooth) approximation $\ph_\kappa$ decreasing pointwise to $\ph$.
\end{thm}

\begin{rem}\label{rk:Wu's-currents-are-the-same}
	In particular, the currents given by \cite{WuX22}*{Thm~2} coincide with the currents defined in \cite{LRRS}.	
\end{rem}

\begin{rem}
	\autoref{thm:wedge-products-of-segre-currents} generalizes \cite{LRRS}*{Thm~1.5} in various ways.
	Among these generalizations, let us highlight that we do not need to assume that on $X\minus\pi(L(\ph))$, $\ph$ is continuous and $\ph_\kappa$ converges locally uniformly to $\ph$.
	This is particularly interesting since the well-known calculus for Chern and Segre forms can be also applied to Chern and Segre currents defined by \eqref{eq:def-SegreCurrents-general} and \eqref{eq:def-ChernCurrents-general} in this general setting, \cf \cite{LRRS}*{Cor.\ 1.9}.
\end{rem}

\medskip

Furthermore, we obtain the following cohomology result.
\begin{thm} 
	In \autoref{st:section-4}
	with compact $X$,
	if $L(\ph)$ is contained in an analytic set of codimension $\geq k_1+\ldots+k_t$, then we get
	\begin{equation*}\begin{split}
		s_{k_t}(E,\ph)\wedge \cdots \wedge s_{k_1}(E,\ph)
		\in\ &s_{k_t}(E)\cdots s_{k_1}(E)\hbox{\quad and}\\
		c_{k_t}(E,\ph)\wedge \cdots \wedge c_{k_1}(E,\ph)
		\in\ &c_{k_t}(E)\cdots c_{k_1}(E).
	\end{split}\end{equation*}
\end{thm}
\begin{proof}
	The argumentation is exactly the same as in the proof of \cite{LRRS}*{Thm\ 1.13} by
	applying Demailly's regularization, see \autoref{thm:Dem-smooth-approximation}.
\end{proof}

\bigskip

\begin{rem-def}\label{rdf:SegreCurrents-analySing}
	Let us assume that the metric $\ph$ on $\cO_\PE(1)$ has analytic singularities.
	Using the generalized Monge-Ampère products, the Segre currents
	\begin{equation}\label{eq:DefAWSegre1F}
		s_k(E,\ph,\alpha):=(-1)^k\pi_\ast\big([dd^c\ph]_\alpha^{k+r-1}\big)
	\end{equation}
	and their wedge products $s_{k_t}(E,\ph,\alpha)\wedge \cdots \wedge s_{k_1}(E,\ph,\alpha) $ can be defined for arbitrary degrees following \cite{LRSW}*{Thm~1.1}.
	If $\alpha$ is a closed $(1,1)$-form representing $\cO_\PE(1)$, we get that these currents represent their corresponding Segre classes. $s_k(E,\ph,\eta)$ can be defined for $\eta$ with analytic singularities in isolated points, as well.
	For every $\xi\in\PE$,
	\autoref{cr:Lelong-estimate-pushforwards-AW-products} implies 
	\begin{equation}
		\nu(s_k(E,\ph,\eta),\pi(\xi))\geq\min\{\nu(\ph,\xi),\nu(\eta,\xi)\}^{k+r-1}
	\end{equation}
	if $k\leq n$ and if $dd^c\ph$ and $\eta$ are positive near $\pi^{-1}(\pi(\xi))$.
	By \autoref{lm:peaker}, we can always find such an $\eta=\eta_\xi$ with $\nu(\eta,\xi)=1$ and neat analytic singularities, see proof of \autoref{MainThm:analytic-singularities}.

	For wedge products of Segre currents, the situation is more complicated.
	Let $\varpi\colon P:=\PE\times_X\cdots\times_X\PE\rightarrow X$ be the $t$-fibre power of $\pi\colon\PE\rightarrow X$, and let $\ph_i$ denote the pullback metric $\pr_i^\ast\ph$ on $\pr_i^\ast\cO_\PE(1)$ where $\pr_i\colon P\rightarrow \PE$ denotes the projection on the $i$th component. For all $k=k_1+\ldots+k_t$, \cite{LRSW} defines
	\begin{equation}\label{eq:DefAWSegreSevF}
		s_{k_t}(E,\ph,\alpha)\wedge \cdots \wedge s_{k_1}(E,\ph,\alpha)
		:=(-1)^{k}
		\varpi_\ast\big([dd^c\ph_t]_{\alpha_t}^{k_t+r-1}\wedge \cdots \wedge [dd^c\ph_1]_{\alpha_1}^{k_1+r-1}\big)
	\end{equation}
	where $\alpha_i=\pr_i^\ast\alpha$ for a closed real $(1,1)$-form $\alpha$ on $\PP(E)$ representing $\cO_\PE(1)$; \cf \autoref{rk:gMAwFP}.
	This cannot be extended straightforwardly to $\eta$ with (neat) analytic singularities in isolated points
	since the pullbacks of such $\eta$ have singularities in sets of lower codimension than points and \BTD cannot be applied in general.
	We guess that it might be possible to still define these products as the singularities of  $\pr_i^\ast\eta$ are transversal to each others.
	An alternative approach could be the following.
	For every $\xi\in\PE$, there is (exactly) one $p\in\bigcap_{i=1}^t \pr_i^{-1}(\xi)$.
	Let $\eta_i$ be a closed quasipositive $(1,1)$-currents with analytic singularities in isolated points including $p$ such that $\eta_i$ is representing $\pr_i^{\ast}\cO_\PE(1)$. %
	Then,
	\begin{equation*}
		S:=%
		\varpi_\ast\big([dd^c\ph_t]_{\eta_t}^{k_t+r-1}\wedge \cdots \wedge [dd^c\ph_1]_{\eta_1}^{k_1+r-1}\big)
	\end{equation*}
	is locally the difference of two closed positive currents and the pushforward of a closed quasipos $(k,k)$-current with analytic singularities, and  $S\in(-1)^{k}s_{k_1}(E)\cdot\ldots\cdot s_{k_t}(E)$. 
	By \autoref{cr:Lelong-estimate-pushforwards-AW-products}, we obtain
	\begin{equation*}
		\nu(S,\pi(\xi))\geq\min\{\nu(\ph,\xi),\nu(\eta_i,p)\}^{k+tr-r}
	\end{equation*}
	if $k\leq n$ and if $dd^c\ph$ positive near $\pi^{-1}(\pi(\xi))$ and $\eta_i$ 
	near $\varpi^{-1}(\pi(\xi))$ since $\nu(\ph,\xi)=\nu(\ph_i,p)$ by \autoref{lm:Lelong-of-pullback}.
	Unfortunately, it is not clear whether such $\eta_i$ exist as $\pr_i^\ast \cO_\PE(1)$ is not Kähler near $\varpi^{-1}(\pi(\xi))$ in general.
\end{rem-def}

\bigskip

Let us conclude this section with the proofs of
\autoref{cr:positive-Segre-current-in-vanishing-class-gives-nef}
and \autoref{thm:vanishing-class-and-more-give-nef}.

\RepeatThirdIntroCor

\begin{proof} 
	We set $k=k_1+\ldots+k_t$.
	Since $S:=(-1)^ks_{k_t}(E,\ph)\wedge \cdots \wedge s_{k_1}(E,\ph)$ is a closed positive current, and since 
	\begin{equation*}
		S\in (-1)^ks_{k_1}(E)\cdots s_{k_t}(E)=0,
	\end{equation*}
	$S$ must vanish. Therefore, $\nu(S,\pi(\xi))=0$, and by \autoref{cr:appl-general-positive}, $\nu(\ph,\xi)=0$ for all $\xi\in\PP(E)$.
	Let $\omega_{\PP(E)}$ denote the Hermitian form on \PE (which is not necessarily closed).
	For every $\eps>0$, there is a singular Hermitian metric $\ph_\eps$ of $\cO_{\PP(E)}(1)$ with analytic singularities such that $dd^c\ph_\eps\geq-2\eps \omega_{\PP(E)}$ and $\nu(\ph_\eps,y)\leq\nu(\ph,y)=0$, see \cite{Demailly92-reg-closed-pos-currents}*{Prop.\ 3.7} or \autoref{thm:DemRegAnalySing}.
	As the Lelong number in an analytic singularity would be positive, we obtain that $L(\ph_\eps)=\emptyset$. By Richberg's approximation \cites{Richberg68}, using the version \cite{Demailly92-reg-closed-pos-currents}*{Lem.\ 2.15}, we can get a smooth approximation $\tilde\ph_\eps$ of $\ph_\eps$ with $dd^c\tilde\ph_\eps\geq-\eps \omega_{\PP(E)}$. As we get these smooth metrics for all $\eps>0$, we conclude $\cO_{\PP(E)}(1)$ is nef, \ie $E$ is nef.
	See also \cite{Boucksom04-Zariski}*{Prop.\ 3.2 (i)}.
\end{proof}

\RepeatFourthIntroCor

\begin{proof}
	As $X$ is Kähler, so is \PE. Let $\omega_\PE$ denote its Kähler form, which is in $\cO_\PE(1)$, and let $r$ denote the rank of $E$. 

	Fix $\eps>0$. 
	Let $H$ be a smooth Hermitian metric on $E$, and let $e^{-\sigma}$ denote the induced metric on $\cO_{\PP(E)}(1)$, \ie $e^{\sigma}=\pi^\ast H|_{\cO_{\PP(E)}(-1)}$. %
	As $dd^c\sigma$ is positive along fibres of $\pi$, there is a small enough $\delta$ such that $dd^c\sigma\geq \delta\omega_\PE-\delta^{-1} \pi^\ast \omega$.
	By Demailly's regularization, see \autoref{thm:DemRegAnalySing}, there exist a singular Hermitian metric $\ph_\eps$ with analytic singularities on $\cO_{\PP(E)}(1)$  such that $dd^c\ph_\eps\geq-\delta^2\eps\omega_\PE$ and $\delta_\eps>0$ (with $\delta_\eps\rightarrow 0$ as $\eps\rightarrow 0$) such that
	\begin{equation*}
		\nu(\ph,\xi)-\delta_\eps\leq\nu(\ph_\eps,\xi)\leq\nu(\ph,\xi)\quad\forall \xi\in \PP(E). %
	\end{equation*}
	As in the proof of \cite{DPS97-nef}*{Thm~11.12}, we can modify $\ph_\eps$ such that $dd^c\ph_\eps\geq-\eps\pi^\ast\omega$:
	Replace $\ph_\eps$ by the barycentre of $\ph_\eps$ and $\sigma$ given by $(1-\delta\eps)\ph_\eps+\delta\eps\cdot\sigma$ with %
	\begin{align*}%
		dd^c\big((1-\delta\eps)\ph_\eps+\delta\eps\sigma)
		&\geq (\delta\eps-1)\delta^2\eps\omega_\PE+\delta\eps(\delta\omega_\PE-\delta^{-1}\pi^\ast\omega)\\
		&= (\delta^3\eps^2-\delta^2\eps+\delta^2\eps)\omega_\PE-\eps \pi^\ast\omega\\
		&\geq -\eps \pi^\ast\omega.%
	\end{align*}%
	Since $\ph_\eps$ has analytic singularities, we get that $L(\ph_\eps)$ is an analytic set.
	Moreover, $\xi\in L(\ph_\eps)$ if and only if $0<\nu(\ph_\eps,\xi)\leq \nu(\ph,\xi)$ such that %
	$\codim \pi({L(\ph_\eps)})\geq k$ by assumption.
	This implies
	\begin{equation*}
		\big(c_1(E,\ph_\eps)\big)^k = \big(-s_1(E,\ph_\eps)\big)^k = \Big(\pi_\ast\big([dd^c\ph_\eps]^r)\Big)^k
	\end{equation*}
	is a closed real $(k,k)$-current in $c_1(E)^k$. %
	As calculated in \autoref{rk:calculations-relative-positive}, we get
	\begin{equation*}
		\pi_\ast\big([dd^c\ph_\eps]^r\big)\wedge T
		= \left(\pi_\ast [dd^c\ph_\eps+\eps\pi^\ast\omega]^{r-1}- r\eps\omega\right)\wedge T.
	\end{equation*}
	Furthermore,
	\begin{equation*}
		 S_\eps:=\Big(\pi_\ast\big([dd^c\ph_\eps+\eps\pi^\ast\omega]^r)\Big)^k
	\end{equation*}
	is a closed positive $(k,k)$-current in $(c_1(E)+r\eps\{\omega\})^k$. %
	So, \autoref{MainThm:general-positive} implies
	\begin{equation}\label{eq:applicationMT}
		\nu(S_\eps,\pi(\xi))\geq\big(\nu(dd^c\ph_\eps+\eps\pi^\ast\omega,\xi)\big)^{kr}=\big(\nu(\ph_\eps,\xi)\big)^{kr}
	\end{equation}
	for all $\xi\in\PP(E)$.

	By weak compactness, there is a subsequence $\eps_\lambda$ such that $S_{\eps_\lambda}$ converges weakly to a closed positive current $S$ in $\big(c_1(E)\big)^k$.
	As assumed, this cohomology vanishes, and so does the closed positive $S$.
	By \autoref{pr:Lelong-estimate-sequence},
	we get that 
	\begin{equation*}
		\lim_{\lambda\rightarrow \infty} \nu(S_{\eps_\lambda},x) \leq \nu(S,0) = 0.
	\end{equation*}
	This in combination with \eqref{eq:applicationMT} and the choice of $\ph_\eps$ implies
	\begin{equation*}
		\nu(\ph,\xi)=\lim_{\lambda\rightarrow \infty} \nu(\ph_{\eps_\lambda},\xi) = 0.
	\end{equation*}
	Repeating the arguments in the proof of \autoref{cr:positive-Segre-current-in-vanishing-class-gives-nef} (or by \cite{Boucksom04-Zariski}*{Prop.~3.2 (i)}), we obtain that $\cO_{\PP(E)}(1)$ is nef.
\end{proof}

\bibliographystyle{amsalpha}

\begin{bibdiv}
\begin{biblist}

	\bib{Andersson05}{article}{
		AUTHOR = {Andersson, Mats},
		 TITLE = {Residues of holomorphic sections and {L}elong currents},
	    JOURNAL = {Ark. Mat.},
	   %
		VOLUME = {43},
		  YEAR = {2005},
		NUMBER = {2},
		 PAGES = {201--219},
		  ISSN = {0004-2080},
	    %
	   %
	 %
	  %
		   URL = {http://dx.doi.org/10.1007/BF02384777},
	 }

\bib{ABW}{article}{
      author={{Andersson}, Mats},
      author={{B\l ocki}, Zbigniew},
      author={{Wulcan}, Elizabeth},
       title={{On a Monge-Amp\`ere operator for plurisubharmonic functions with
  analytic singularities}},
        date={2019},
        ISSN={0022-2518},
     journal={{Indiana Univ. Math. J.}},
      volume={68},
      number={4},
       pages={1217\ndash 1231},
}

\bib{ASWY17}{article}{
    AUTHOR = {{A}ndersson, Mats},
author={{S}amuelsson {}Kalm, H\aa kan},
author={{W}ulcan, Elizabeth},
author={{Y}ger, Alain},
     TITLE = {Segre numbers, a generalized {K}ing formula, and local
              intersections},
   JOURNAL = {J. Reine Angew. Math.},
  %
  %
    VOLUME = {728},
      YEAR = {2017},
     PAGES = {105--136},
%
%
%
       URL = {https://doi.org/10.1515/crelle-2014-0109},
}

\bib{AW14-MA}{article}{
      author={Andersson, Mats},
      author={Wulcan, Elizabeth},
       title={Green functions, {Segre} numbers, and {King{\textquoteright}s}
  formula},
        date={2014},
     journal={Annales de l'Institut Fourier},
      volume={64},
      number={6},
       pages={2639\ndash 2657},
         url={https://aif.centre-mersenne.org/articles/10.5802/aif.2922/},
}

\bib{Bingener83-II}{article}{
 Author = {Bingener, J{\"u}rgen},
 Title = {On deformations of {K{\"a}hler} spaces. {II.}},
 FJournal = {Archiv der Mathematik},
 Journal = {Arch. Math.},
 ISSN = {0003-889X},
 Volume = {41},
 Pages = {517--530},
 Year = {1983},
 %
 %
 Keywords = {32G07,32H35,32J25,58H15},
 zbMATH = {3935471},
 Zbl = {0584.32043}
}

\bib{Blocki19-quasiMA}{article}{
      author={{B{\l}ocki}, Zbigniew},
       title={{On the complex Monge-Amp\`ere operator for
  quasi-plurisubharmonic functions with analytic singularities}},
        date={2019},
        ISSN={0024-6093},
     journal={{Bull. Lond. Math. Soc.}},
      volume={51},
      number={3},
       pages={431\ndash 435},
}

\bib{Boucksom04-Zariski}{article}{
      author={{Boucksom}, S\'ebastien},
       title={{Divisorial Zariski decompositions on compact complex
  manifolds}},
        date={2004},
        ISSN={0012-9593},
     journal={{Ann. Sci. \'Ec. Norm. Sup\'er. (4)}},
      volume={37},
      number={1},
       pages={45\ndash 76},
}

\bib{BDPP13}{article}{
      author={{Boucksom}, S\'ebastien},
      author={{Demailly}, Jean-Pierre},
      author={{P\u{a}un}, Mihai},
      author={{Peternell}, Thomas},
       title={{The pseudo-effective cone of a compact K\"ahler manifold and
  varieties of negative Kodaira dimension}},
        date={2013},
        ISSN={1056-3911},
     journal={{J. Algebr. Geom.}},
      volume={22},
      number={2},
       pages={201\ndash 248},
}

\bib{BP08}{article}{
      author={{Berndtsson}, Bo},
      author={{P\u{a}un}, Mihai},
       title={{Bergman kernels and the pseudoeffectivity of relative canonical
  bundles}},
        date={2008},
        ISSN={0012-7094},
     journal={{Duke Math. J.}},
      volume={145},
      number={2},
       pages={341\ndash 378},
}

\bib{BT76}{article}{
      author={{Bedford}, Eric},
      author={{Taylor}, B.~A.},
       title={{The Dirichlet problem for a complex Monge-Amp\`ere equation}},
        date={1976},
        ISSN={0020-9910},
     journal={Invent. Math.},
      volume={37},
      number={1},
       pages={1\ndash 44},
}

\bib{BT82}{article}{
      author={{Bedford}, Eric},
      author={{Taylor}, B.~A.},
       title={{A new capacity for plurisubharmonic functions}},
        date={1982},
        ISSN={0001-5962},
     journal={Acta Math.},
      volume={149},
      number={1-2},
       pages={1\ndash 40},
}





\bib{Demailly82}{article}{
      author={Demailly, Jean-Pierre},
       title={Estimations {$L\sp{2}$} pour l'op\'{e}rateur {$\bar \partial $}
  d'un fibr\'{e} vectoriel holomorphe semi-positif au-dessus d'une
  vari\'{e}t\'{e} k\"{a}hl\'{e}rienne compl\`ete},
        date={1982},
        ISSN={0012-9593},
     journal={Ann. Sci. \'{E}cole Norm. Sup. (4)},
      volume={15},
      number={3},
       pages={457\ndash 511},
         url={http://www.numdam.org/item?id=ASENS_1982_4_15_3_457_0},
}

\bib{Demailly85}{book}{
      author={Demailly, Jean-Pierre},
       title={Mesures de {Monge-Amp\`ere} et caract\'erisation g\'eom\'etrique
  des vari\'et\'es alg\'ebriques affines},
      series={M\'emoires de la Soci\'et\'e Math\'ematique de France},
   publisher={Soci\'et\'e math\'ematique de France},
        date={1985},
      number={19},
         url={http://www.numdam.org/item/MSMF_1985_2_19__1_0/},
     %
}

\bib{Demailly92-reg-closed-pos-currents}{article}{
      author={Demailly, Jean-Pierre},
       title={{Regularization of closed positive currents and intersection
  theory}},
        date={1992},
        ISSN={1056-3911},
     journal={{J. Algebr. Geom.}},
      volume={1},
      number={3},
       pages={361\ndash 409},
}

\bib{Demailly93}{incollection}{
      author={Demailly, Jean-Pierre},
       title={Monge-{A}mp\`ere operators, {L}elong numbers and intersection
  theory},
        date={1993},
   booktitle={In: Complex analysis and geometry, Plenum Press.},
      series={Univ. Ser. Math.},
   publisher={Plenum, New York},
       pages={115\ndash 193},
}

\bib{DemaillyAG}{unpublished}{
      author={Demailly, Jean-Pierre},
       title={{Complex Analytic and Differential Geometry}},
        date={\noopsort{1999z}2012},
  url={\url{http://www-fourier.ujf-grenoble.fr/~demailly/manuscripts/agbook.pdf}},
        note={Institut Fourier, Universit\'e de Grenoble I. OpenContent
  AG-Book},
}

\bib{DPS97-nef}{article}{
      author={Demailly, Jean-Pierre},
      author={{Peternell}, Thomas},
      author={{Schneider}, Michael},
       title={{Compact complex manifolds with numerically effective tangent
  bundles}},
        date={1994},
        ISSN={1056-3911},
     journal={{J. Algebr. Geom.}},
      volume={3},
      number={2},
       pages={295\ndash 345},
}

\bib{FischerGrauert65}{article}{
 Author = {Fischer, Wolfgang},
 Author = {Grauert, Hans},
 Title = {Lokal-triviale {Familien} kompakter komplexer {Mannigfaltigkeiten}},
 FJournal = {Nachrichten der Akademie der Wissenschaften in G{\"o}ttingen. II. Mathematisch-Physikalische Klasse},
 Journal = {Nachr. Akad. Wiss. G{\"o}tt., II. Math.-Phys. Kl.},
 ISSN = {0065-5295},
 Volume = {1965},
 Pages = {89--94},
 Year = {1965},
 %
 zbMATH = {3219133},
 Zbl = {0135.12601}
}


\bib{Fujiki78}{article}{
 Author = {Fujiki, Akira},
 Title = {Closedness of the {Douady} spaces of compact {K{\"a}hler} spaces},
 FJournal = {Publications of the Research Institute for Mathematical Sciences, Kyoto University},
 Journal = {Publ. Res. Inst. Math. Sci.},
 ISSN = {0034-5318},
 Volume = {14},
 Pages = {1--52},
 Year = {1978},
 %
 %
 Keywords = {32J25,32C25,32C15,53C55,32C37},
 zbMATH = {3636390},
 Zbl = {0409.32016}
}


\bib{Lelong57}{article}{
      author={Lelong, Pierre},
       title={Int\'egration sur un ensemble analytique complexe},
        date={1957},
     journal={Bulletin de la Soci\'et\'e Math\'ematique de France},
      volume={85},
       pages={239\ndash 262},
         url={http://www.numdam.org/articles/10.24033/bsmf.1488/},
}

\bib{Lelong68}{book}{
      author={Lelong, Pierre},
       title={{Fonctions plurisousharmoniques et formes diff\'{e}rentielles
  positives}},
   publisher={Gordon \& Breach, Paris-London-New York (distributed by Dunod
  \'{E}diteur, Paris)},
        date={1968},
}

\bib{LRRS}{article}{
      author={L\"{a}rk\"{a}ng, Richard},
      author={Raufi, Hossein},
      author={Ruppenthal, Jean},
      author={Sera, Martin},
       title={Chern forms of singular metrics on vector bundles},
        date={2018},
     journal={Adv. Math.},
      volume={326},
       pages={465\ndash 489},
}

\bib{LRSW}{article}{
      author={{L{\"a}rk{\"a}ng}, Richard},
      author={{Raufi}, Hossein},
      author={{Sera}, Martin},
      author={{Wulcan}, Elizabeth},
       title={{Chern forms of Hermitian metrics with analytic singularities on
  vector bundles}},
        date={2022},
        ISSN={0022-2518},
     journal={Indiana Univ. Math. J.},
      volume={71},
      number={1},
       pages={153\ndash 189},
         url={https://doi.org/10.1512/iumj.2022.71.8834},
}

\bib{LSW}{article}{
      author={{L\"ark\"ang}, Richard},
      author={{Sera}, Martin},
      author={{Wulcan}, Elizabeth},
       title={{On a mixed Monge-Amp\`ere operator for quasiplurisubharmonic
  functions with analytic singularities}},
        date={2020},
        ISSN={0024-6093},
     journal={{Bull. Lond. Math. Soc.}},
      volume={52},
      number={1},
       pages={77\ndash 93},
}

\bib{Meo96}{article}{
      author={Meo, Michel},
       title={Image inverse d'un courant positif ferm\'{e} par une application
  analytique surjective},
        date={1996},
        ISSN={0764-4442},
     journal={C. R. Acad. Sci. Paris S\'{e}r. I Math.},
      volume={322},
      number={12},
       pages={1141\ndash 1144},
}

\bib{Richberg68}{article}{
      author={Richberg, Rolf},
       title={Stetige streng pseudokonvexe {F}unktionen},
        date={1968},
        ISSN={0025-5831},
     journal={Math. Ann.},
      volume={175},
       pages={257\ndash 286},
         url={https://doi.org/10.1007/BF02063212},
}

\bib{WuX22}{article}{
      author={{Wu}, Xiaojun},
       title={Pseudo-effective and numerically flat reflexive sheaves},
        date={2022},
        ISSN={1050-6926},
     journal={J. Geom. Anal.},
      volume={32},
      number={4},
       pages={article 124},
         url={https://doi.org/10.1007/s12220-021-00865-0},
}

\end{biblist}
\end{bibdiv}
\newcommand{\noopsort}[1]{}

\end{document}